%% file: main.tex
\documentclass[a4paper,reqno]{amsart}

\usepackage[a4paper,left=3cm,right=3cm,top=3cm,bottom=3cm]{geometry}

\usepackage{graphicx} 
\usepackage{tikz}
\usetikzlibrary{patterns,shapes,decorations.pathmorphing,decorations.pathreplacing,calc,arrows,cd}
\usetikzlibrary{tikzmark, fit, matrix, decorations.markings}
\usepackage{mathdots}
\usepackage{amssymb}
\usepackage{amstext}
\usepackage{amsmath}
\usepackage{amscd}
\usepackage{amsthm}
\usepackage{amsfonts}
\usepackage{comment}

\usepackage{array}
\usepackage{enumerate}
\usepackage{latexsym}
\usepackage{mathrsfs}
\usepackage[hidelinks]{hyperref}

\usepackage[all,color]{xy}
\usepackage{mathtools}

\theoremstyle{plain}
\newtheorem{theorem}{Theorem}[section]
\newtheorem{proposition}[theorem]{Proposition}
\newtheorem{definition}[theorem]{Definition}
\newtheorem{corollary}[theorem]{Corollary}
\newtheorem{lemma}[theorem]{Lemma}
\theoremstyle{definition}
\newtheorem{remark}[theorem]{Remark}
\DeclareMathOperator{\vect}{\mathsf{vect}}

\DeclareMathOperator{\Vect}{\mathsf{Vect}}
\DeclareMathOperator{\colim}{\mathsf{colim}}
\DeclareMathOperator{\Ch}{\mathsf{Ch}}
\DeclareMathOperator{\Chnn}{\mathsf{Ch}_{\geq 0}}

\DeclareMathOperator{\Cone}{\mathrm{Cone}}
\DeclareMathOperator{\Mor}{\mathrm{Mor}}
\DeclareMathOperator{\Var}{\mathrm{Var}}
\newcommand{\Fun}{\mathrm{Fun}}
\DeclareMathOperator{\Ker}{\mathrm{Ker}}
\DeclareMathOperator{\Image}{Im}
\DeclareMathOperator{\Coker}{\mathrm{Coker}}

\DeclareMathOperator{\cI}{\mathcal{I}}
\DeclareMathOperator{\NN}{\mathbb{N}}
\DeclareMathOperator{\EE}{\mathbb{E}}
\DeclareMathOperator{\NNn}{\mathbb{N}}
\newcommand{\RR}{\mathbb{R}}
\newcommand{\ZZ}{\mathbb{Z}}
\newcommand{\cU}{{\mathcal{F}(\RR^d)}}
\newcommand{\add}{u^+}
\newcommand{\cR}{\mathcal{R}}
\DeclareMathOperator{\cA}{\mathcal{A}}
\DeclareMathOperator{\cP}{\mathcal{P}}
\DeclareMathOperator{\cB}{\mathcal{B}}

\DeclareMathOperator{\cW}{\mathcal{W}}
\DeclareMathOperator{\cM}{\mathcal{M}}
\DeclareMathOperator{\id}{\mathrm{id}}
\DeclareMathOperator{\rank}{\mathrm{rank}}

\DeclareMathOperator{\lf}{\mathcal{F}_{\mathrm{lf}}(\mathbb{R}^d)}


\begin{document}

\title[Central Limit Theorems for Persistent Betti Numbers]{Central Limit Theorems for Persistent Betti Numbers}

\author[Shunsuke Tada]{Shunsuke Tada}
\address{Shunsuke Tada, Mathematical Science Center for Co-Creative Society, Tohoku University, 1-1, Katahira 2-chome Aoba-ku, Sendai, Miyagi 980-0812, Japan}
\email{shunsuke.tada.e6@tohoku.ac.jp}
\date{}

\begin{abstract}
Various persistent invariants have been developed in recent years. In this paper, we develop a method based on homological algebra for deriving central limit theorems for persistent Betti numbers associated with $\mathbb{R}$-indexed chain complexes constructed from a homogeneous Poisson point process of unit intensity on $\mathbb{R}^d$. Our method also applies to Gibbs point processes.
This algebraic method reduces the verification of the required conditions to checking properties of the kernels and cokernels of add one maps between chain complexes. As an application, we recover the known central limit theorem for persistent Betti numbers arising from simplicial complex filtrations. In addition, we prove a central limit theorem for persistent Betti numbers of $\ell_p$-Vietoris-Rips simplicial homology, including blurred magnitude homology of random geometric graphs. We also establish central limit theorems for relative $\ell_p$-Vietoris-Rips homology and, in particular, for magnitude Betti numbers of random geometric graphs. These results suggest that our algebraic method can be applied more broadly to derive central limit theorems for persistent Betti numbers.
\end{abstract}

\maketitle

\setcounter{tocdepth}{1}
\tableofcontents
\input{1_Introduction}
\input{2_Preliminaries}

\input{3_algebra_main}

\input{3.5_addone}

\input{5_Criterion}
\input{6_Examples}

\input{9_Appendix}

\section*{Acknowledgements.}

The author would like to thank Shunsuke Kano for raising the question of central limit theorems for magnitude homology and for suggesting that the class of applications be broadened from blurred magnitude homology
to \(\ell_p\)-Vietoris-Rips simplicial homology. 

This work was supported by JST CREST, Japan, Grant Number JPMJCR24Q6, and by JSPS KAKENHI Grant Number JP25K23328.

\section*{Use of generative AI.}
In writing this paper, the author used ChatGPT (GPT-5.6 Sol; OpenAI) to assist with English-language editing, the
organization of the exposition, and the exploration and preliminary
checking of mathematical arguments and proof strategies. The tool was also used to identify points requiring further clarification or verification in preliminary drafts. All AI-assisted text was independently checked and revised by the author. The author assumes full responsibility for the text of the paper.

\bibliographystyle{alpha} 
\bibliography{main.bib}
\end{document}

%% file: 1_Introduction.tex
\section{Introduction}
Various persistent homological invariants have been developed in recent years.
The most prominent example is persistent homology
\cite{EL02,carlsson2005computing}, one of the main tools in topological
data analysis \cite{Wa18}. It has found applications in a wide range of fields \cite{hiraoka2016hierarchical,CCR13, CM18,KOHY23}. 
More recently, several persistent invariants have also been introduced, including blurred magnitude homology \cite{Ott22}, persistent path homology \cite{CM18}, and persistent magnitude homology \cite{BFLW26}; see also \cite{MR20,F22,CR24}. 

A basic problem in random topology is to study topological invariants associated with random point configurations and with the graphs and simplicial complexes they generate \cite{Penrose2003,BobrowskiKahle2018}.
In the probabilistic study of persistent homology, the Poisson central limit theorem for ordinary Betti numbers established in \cite{YSEA17} was extended to persistent Betti numbers in \cite{HST18}.
Later studies established functional central limit theorems \cite{KH22} and multivariate central limit theorems \cite{KP25} for persistent Betti numbers; see also \cite{JKP26,BH24,HOS25,HOS26}.
A feature of these approaches is that the central limit theorem is reduced to stabilization of the effect of a local modification of the underlying point process.

Existing stabilization arguments for persistent Betti numbers often rely on two properties: the chain maps induced by adding a point, which we call \emph{add one maps} (see Section~\ref{sec:add one-stabilization}), are injective, and adding a point does not change the metric structure on the original point set (see, for example, \cite[Lemma~5]{KP25} and \cite[Lemma~5.3]{HST18}). These properties hold for ordinary simplicial complex filtrations but may fail for other natural constructions. For instance, adding a vertex to a random geometric graph may shorten the graph distances between existing vertices. Moreover, for relative constructions, including persistent magnitude homology of random geometric graphs, the resulting add one maps between chain complexes need not be injective. Consequently, the preceding stabilization arguments are no longer directly applicable.


The aim of this paper is to provide a unified method for deriving central limit theorems for persistent Betti numbers arising from
various homology theories.

Let \(\cU\) denote the category whose objects are the finite subsets
of \(\RR^d\) and whose morphisms are inclusions. Let 
\[
\mathcal{A}\colon
\cU
\longrightarrow
\Fun
\bigl(
\mathbb{R},
\Chnn(\vect_k)
\bigr)
\]
be a functor, and fix \(q\in\NN\) and \(r,s\in\mathbb{R}\)
with \(r\leq s\). The \(q\)th \((r,s)\)-persistent Betti number functional
associated with \(\mathcal{A}\) is
\begin{equation}\label{eq:PBN}
\beta_q^{r,s}(\mathcal{A})(F)
:=
\rank
\left(
H_q(\mathcal{A}(F)(r))
\rightarrow
H_q(\mathcal{A}(F)(s))
\right).    
\end{equation}
Here, \(\cA\) should be viewed as an abstract model for a construction
that assigns an \(\RR\)-indexed chain complex to each finite point
cloud. For example, this setting includes chain complexes of ordinary simplicial complex filtrations and of \(\ell_p\)-Vietoris-Rips simplicial sets, including those underlying
blurred magnitude homology, as well as relative constructions such as those underlying persistent magnitude homology. 

To state our main result, we introduce the following notion. For $t \in \RR,$ we say that \(\mathcal{A}\) is \emph{admissible at \(t\)} if it satisfies the following conditions:
\begin{enumerate}
\item
\(\mathcal{A}\) is translation invariant.
\item
\(\mathcal{A}\) satisfies \emph{add one kernel-cokernel stabilization at
\(t\)} (Definition~\ref{def:add one-kernel-cokernel-stabilization}).
\item
\(\mathcal{A}\) satisfies \emph{add one kernel-cokernel boundedness} (Definition~\ref{def:boundedness}).
\end{enumerate}
Together with the measurability of the associated persistent Betti number functional, these conditions imply the translation invariance, stabilization, and moment assumptions in the central limit theorem of Penrose and Yukich \cite[Theorem~3.1]{PY01}. This leads to one of the main results of this paper.

\begin{theorem}[Theorem~\ref{thm:clt_main}]
\label{thm:clt-main} Fix \(q\in\mathbb N\) and \(r,s\in\mathbb R\) with \(r\le s\).
Let
$
\mathcal{A}\colon
\cU
\to
\Fun
\bigl(
\mathbb{R},
\Chnn(\vect_k)
\bigr)
$ be admissible at \(r\) and \(s\). Assume  that  the functional
$
\beta_q^{r,s}(\cA )
$
is measurable.

Let \(\mathcal{P}\) be a homogeneous Poisson point process on \(\RR^d\) with unit intensity. Then,  there exists a constant $
\sigma^2_{q,r,s}\in[0,\infty)
$ such that 
\[
\frac{\Var\!\left[
\beta_q^{r,s}(\cA)
(\mathcal{P}_{\Lambda_n})
\right]}{n}
\longrightarrow
\sigma_{q,r,s}^2 \quad
\text{ and }\quad 
\frac{
\beta_q^{r,s}(\cA)
\bigl(\mathcal{P}_{\Lambda_n}\bigr)
-
\mathbb{E}\!\left[
\beta_q^{r,s}(\cA)
(\mathcal{P}_{\Lambda_n})
\right]
}{
\sqrt{n}
}
\xrightarrow{d}
\mathcal{N}(0,\sigma_{q,r,s}^2)
\]
as \(n\to\infty\).

Here,
$
\Lambda_n
:=
\left[
-\frac{1}{2}n^{1/d},
\frac{1}{2}n^{1/d}
\right)^d$ and 
$
\mathcal{P}_{\Lambda_n}
$
denote the restriction of $\cP$ to $\Lambda_n$. 
Moreover, \(\mathcal{N}(\mu,\sigma^2)\) denotes the normal distribution
with mean \(\mu\) and variance $\sigma^2$, and
\(\xrightarrow{d}\) denotes convergence in distribution.
\end{theorem}

Similarly, with the central limit theorem of Hirsch, Otto, and Svane \cite[Theorem~3]{HOS25}, we obtain the following.

\begin{theorem}[Theorem~\ref{thm:gibbs-clt-main}]
\label{thm:gibbs-clt_main} Fix \(q\in\mathbb N\) and \(r,s\in\mathbb R\) with \(r\le s\).
Let $
\cA\colon
\mathcal{F}(\mathbb{R}^{d})
\rightarrow
\Fun(
 \mathbb{R},
 \Chnn(\vect_{k}))
$ be admissible at $r$ and $s$. Assume that
$\beta_{q}^{r,s}(\cA)$ is measurable.

Let $\mathcal{X}$ be an infinite-volume Gibbs point process with
translation-invariant Papangelou conditional intensity satisfying condition
\eqref{eq:gibbs-process-assumption}. Then there exists
$
\sigma_{q,r,s}^{2}\in[0,\infty)$ such that
\[
\frac{
\operatorname{Var}\left[
 \beta_{q}^{r,s}(\cA)(\mathcal{X}_{Q_{n}})
\right]}{ n^{d}}
\to
\sigma_{q,r,s}^{2}
\text{ and }
\frac{\left(
 \beta_{q}^{r,s}(\cA)(\mathcal{X}_{Q_{n}})
 -
 \mathbb{E}\left[
  \beta_{q}^{r,s}(\cA)(\mathcal{X}_{Q_{n}})
 \right]
\right)}{n^{d/2}}
\xrightarrow{d}
\mathcal{N}(0,\sigma_{q,r,s}^{2})\]
as $n\to\infty$, where $Q_n :=[-n/2, n/2 ]^d.$
\end{theorem}

Before turning to applications of our theorems, we note several
advantages of our algebraic formulation.
\begin{itemize}
\item Our method does not require add one maps to be injective or adding a point to preserve the metric structure on the original point set. 

\item Conditions~(2) and~(3) are formulated in terms of the kernels and cokernels of add one maps at the chain level. This means that weak stabilization and bounded moment conditions can be verified without explicitly tracking the cycle or boundary groups.

\item Admissibility is preserved under standard constructions in an abelian category, including kernels, cokernels, and images of structure morphisms; see Proposition~\ref{prop:relative}(3) for details. 
In particular, this closure property allows us to apply our theorem to relative persistent Betti number functionals. 
\end{itemize}

We now turn to applications of the theorem to several homology theories.
We first recover the known central limit
theorem for persistent Betti numbers established in \cite[Theorem~5.2]{HST18}.  
We then obtain central limit theorems for persistent Betti numbers arising from the \(\ell_p\)-Vietoris-Rips homology and from the relative \(\ell_p\)-Vietoris-Rips homology  of Poisson random geometric graphs. These applications include blurred magnitude homology and  persistent magnitude homology of Poisson random geometric graphs. As an illustration, Figure~\ref{fig:clt-mag} presents the magnitude Betti numbers in bidegree $(1,1)$.

\begin{figure}[htbp]
  \centering
  \includegraphics[width=15cm]{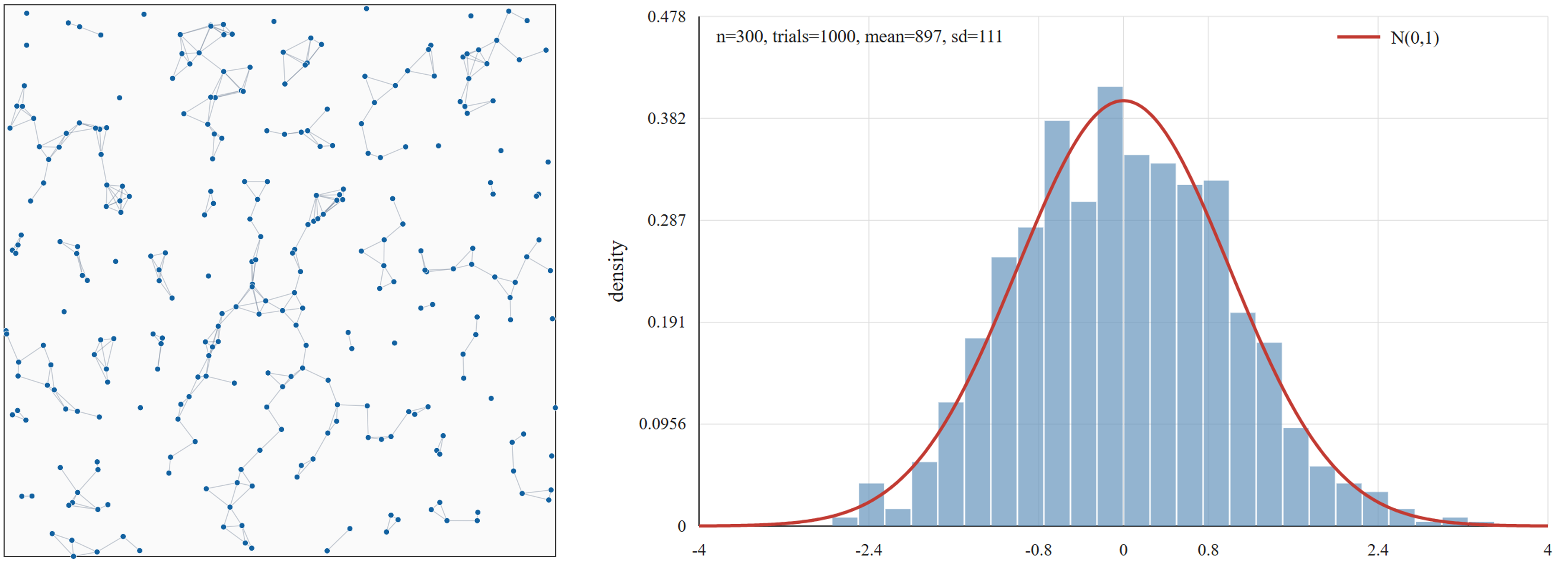} 
 \caption{A numerical illustration of
Corollary~\ref{cor:clt_PMG_graph} for magnitude homology of random geometric graphs. Left: a Poisson random geometric graph in the observation window \(\Lambda_{300}\). Right: the standardized distribution of
the magnitude Betti number in bidegree $(1,1)$,
based on \(1000\) independent trials. The solid curve represents the density of the standard normal
distribution.}
  \label{fig:clt-mag}
\end{figure}

We also discuss extended persistent Betti numbers introduced in \cite{JKP26}. If
$
\cA\colon
\cU
\rightarrow
\Fun(\mathbb{R},\Chnn(\vect_k))
$
is pointwise admissible, then so
is the transformed functor  defining the extended
persistent Betti numbers; see
Proposition~\ref{prop:extended-preservation}.

\medskip

\medskip

\noindent\textbf{Notation.}
Throughout this paper, \(k\) denotes a field. We write \(\Vect_k\) for
the category of \(k\)-vector spaces and \(\vect_k\) for its full
subcategory of finite-dimensional \(k\)-vector spaces. We denote by
\(\Ch(\Vect_k)\) the category of chain complexes of \(k\)-vector spaces
and by \(\Chnn(\Vect_k)\) its full subcategory consisting of chain
complexes concentrated in nonnegative degrees. Similarly,
\(\Ch(\vect_k)\) denotes the full subcategory of \(\Ch(\Vect_k)\)
consisting of degreewise finite-dimensional chain complexes, and
\(\Chnn(\vect_k)\) denotes its full subcategory consisting of chain
complexes concentrated in nonnegative degrees.
We write \(\NN:=\{0,1,2,\ldots\}\) for the set of nonnegative integers.

%% file: 2_Preliminaries.tex
\section{Algebraic preliminaries}
\label{sec:algebraic-preliminaries}

In this section, we recall the basics of  homological algebra and category theory used throughout the
paper see; \cite{Weibel1994} and \cite{MR98} for details.

\subsection{Filtered colimits}
\label{subsec:filtered-colimits}

For categories $\mathcal{C}$ and $\mathcal{D}$, we write
$\Fun(\mathcal{C},\mathcal{D})$ for the functor category. If $\mathcal{D}$ is an abelian category, then $\Fun(\mathcal{C},\mathcal{D})$ is also abelian, and the image, kernel, and cokernel of a morphism in $\Fun(\mathcal{C},\mathcal{D})$  are computed objectwise. The arrow category of $\mathcal{C}$ is $\Mor(\mathcal{C}):= \Fun(\mathsf{2},\mathcal{C})$, where \(\mathsf{2}\) is the category associated with the ordered set \(0<1\). That is, the objects of this category are morphisms in
$\mathcal{C}$ and a morphism from $f\colon X\to Y$ to $g\colon Z\to W$ is a
pair $(u,v)$ fitting into a commutative diagram
\[
\begin{CD}
X @>{f}>> Y\\
@V{u}VV @VV{v}V\\
Z @>{g}>> W.
\end{CD}
\]

A \emph{filtered poset} is a nonempty poset $(\cI,\leq)$ such that every pair $A,B\in\cI$ admits a common upper bound. We regard $\cI$ as a category with a unique morphism
$A\to B$ whenever $A\leq B$. A functor from $\cI$ to a category $\mathcal{C}$ will be called an $\cI$-indexed diagram in $\mathcal{C}$. Let $\cI$ be a filtered poset. The filtered colimit functor
$
\colim \colon
\Fun(\cI,\Vect_k)
\to
\Vect_k
$
preserves colimits and finite limits. In particular, it is exact.

\subsection{Chain complexes, exact sequences, and mapping cones}
\label{subsec:chain-complexes-and-cones}

For a chain complex $C\in\Ch(\Vect_k)$, write
$
\partial_q^C\colon C_q\to C_{q-1}
$
for its differential, and set
$
Z_q(C):=\Ker\partial_q^C$, $B_q(C):=\Image \partial_{q+1}^C$, $H_q(C):=Z_q(C)/B_q(C)$.
For a chain map $f\colon C\to D$, the induced map on homology is denoted by
$
H_q(f)\colon H_q(C)\to H_q(D).
$
Let
$
0\to C\xrightarrow{}D\xrightarrow{}E\to 0
$ be a short exact sequence of chain complexes. We denote the
connecting homomorphism in the associated long exact sequence by
$
\delta_q\colon H_q(E)\to H_{q-1}(C).
$

\begin{lemma}
\label{lem:dimension-identities-long-exact-sequence}
Let $
0\to C\xrightarrow{i}D\xrightarrow{p}E\to 0$
be a short exact sequence of chain complexes, and let $q\in\mathbb{Z}$. Assume
that all the homology groups appearing in the formulas below are
finite-dimensional. Then
\[
\begin{aligned}
\dim_k\Image  H_q(i)
=
\dim_k H_q(D)+\dim_k H_{q-1}(C)-\dim_k H_q(E)-\dim_k\Image  H_{q-1}(i),
\end{aligned}
\]
and
\[
\dim_k H_q(D)-\dim_k H_q(C)
=
\dim_k\Ker\delta_q+\dim_k\Ker\delta_{q+1}
-\dim_k H_{q+1}(E).
\]
\end{lemma}

\begin{proof}
Both equalities follow by taking dimensions in the exact sequence
\[
H_{q+1}(E)\xrightarrow{\delta_{q+1}}H_q(C)
\xrightarrow{H_q(i)}H_q(D)\xrightarrow{H_q(p)}H_q(E)
\xrightarrow{\delta_q}H_{q-1}(C)
\xrightarrow{H_{q-1}(i)}H_{q-1}(D).
\]
\end{proof}

For a chain map $f\colon C\to D$, its mapping cone is the chain complex
$\Cone (f)$ given by
\[
\Cone (f)_q:=D_q\oplus C_{q-1}, \quad 
\partial_q^{\Cone (f)}(y,x)
:=
(\partial_q^D (y)+f_{q-1}(x),-\partial_{q-1}^C (x)).
\]
Then, there is a short exact sequence
$
0\to D\to\Cone (f)
\to C[1]\to 0,
$
where we put $
C[1]_q:=C_{q-1}$, $
\partial_q^{C[1]}:=-\partial_{q-1}^C.
$ For this sequence, we have the long exact sequence
\begin{equation}
\label{eq:mapping-cone-long-exact-sequence}
\begin{aligned}
\cdots\to
H_q(C)\xrightarrow{H_q(f)}H_q(D)
&\to H_q(\Cone (f))
\to H_{q-1}(C)\\
&\xrightarrow{H_{q-1}(f)}H_{q-1}(D)
\to\cdots.
\end{aligned}
\end{equation}
Note that \(H_q(\Cone (f))\) is finite-dimensional
whenever \(H_q(D)\) and \(H_{q-1}(C)\) are finite-dimensional, since
the long exact sequence yields
\begin{equation}\label{eq:fin-dim-cone}
\dim_k H_q(\Cone (f))
\leq
\dim_k H_q(D)+\dim_k H_{q-1}(C).    
\end{equation}

\begin{lemma}
\label{lem:mapping-cone-rank-identity}
Let $f\colon C\to D$ be a chain map, and let $q\in\mathbb{Z}$. Assume that all
the homology groups appearing in the formula below are finite-dimensional.
Then
\[
\begin{aligned}
\dim_k\Image  H_q(f)
={}&
\dim_k H_q(D)+\dim_k H_{q-1}(C)\\
&-\dim_k H_q(\Cone (f))
-\dim_k\Image  H_{q-1}(f).
\end{aligned}
\]
\end{lemma}

\begin{proof}
For a chain map $f\colon C\to D$, we have the long exact sequence~\eqref{eq:mapping-cone-long-exact-sequence}. By Lemma~\ref{lem:dimension-identities-long-exact-sequence}, we obtain the equality.
\end{proof}



The mapping cone construction is functorial. More precisely,  a
morphism $(u,v)\colon f\to g$ in
$\Mor(\Ch(\Vect_k))$ induces the chain map
\[
\Cone (u,v) 
\colon
\Cone (f)\to\Cone (g),
\qquad
\Cone (u,v)_q
:= \begin{pmatrix}
   v_q & 0 \\
   0 &  u_{q-1}
\end{pmatrix}.
\]
It gives rise to the functor
$
\Cone 
\colon
\Mor(\Ch(\Vect_k))
\to
\Ch(\Vect_k)
$. By construction, this functor is exact. In particular, it preserves monomorphisms and epimorphisms.

Let $(u,v)\colon f\to g$ be a morphism in
$\Mor(\Ch(\Vect_k))$, and put
$
c=\Cone (u,v)$.
The commutative square defining $(u,v)$ induces chain maps
$
\widehat{f}\colon\Ker (u)\to\Ker (v)$ and $ \overline{g}\colon\Coker(u)\to
\Coker (v).
$
Then, there exist natural isomorphisms of chain complexes
\begin{equation}\label{eq:cone_cokernel}
\Ker(c)\cong\Cone (\widehat{f}),
\qquad
\Coker(c)\cong\Cone (\overline{g}).    
\end{equation}
The induced maps and the above isomorphisms fit into the following commutative diagram:
\[
\begin{tikzcd}[column sep=large, row sep=large]
\ker(u)
  \arrow[r, hook]
  \arrow[d, "\widehat{f}"']
&
C
  \arrow[r, "u"]
  \arrow[d, "f"']
&
C'
  \arrow[r, two heads]
  \arrow[d, "g"']
&
\Coker(u)
  \arrow[d, "\overline{g}"]
\\
\ker(v)
  \arrow[r, hook]
  \arrow[d]
&
D
  \arrow[r, "v"]
  \arrow[d]
&
D'
  \arrow[r, two heads]
  \arrow[d]
&
\Coker(v)
  \arrow[d]
\\
\Cone (\widehat{f})\cong\ker(c)
  \arrow[r]
&
\Cone (f)
  \arrow[r, "c"]
&
\Cone (g)
  \arrow[r, two heads]
&
\Cone (\overline{g})\cong\Coker(c).
\end{tikzcd}
\]

%% file: 3_algebra_main.tex
\section{Kernel-cokernel stabilization and rank differences}\label{sec:alg_main}


Let $\cI$ be a filtered poset.
The following notions play a central role in this paper.

\begin{definition}\label{def:stab}
We say that an object
$
C\in
\Fun
(\mathcal{I},\Ch(\Vect_k))
$ \emph{stabilizes degreewise} if, for every $q \in \ZZ$,  there exists \(A_q\in\mathcal{I}\)
such that, for every \(A,B\in\mathcal{I}\) satisfying
$
A_q\leq A\leq B,
$
the structure morphism
$
(C_A)_q\to (C_B)_q
$
is an isomorphism.

We say that 
a morphism in
$
\Fun
(\mathcal{I},\Ch(\Vect_k))
$
satisfies \emph{kernel stabilization}
(resp.\ \emph{cokernel stabilization}) if
\(\Ker(u)\) (resp.\ \(\Coker(u)\)) stabilizes degreewise. We say that a morphism satisfies \emph{kernel-cokernel stabilization}
if it satisfies both kernel stabilization and cokernel
stabilization.
\end{definition}

This section aims to prove the following.

\begin{theorem}
\label{cor:main_result}
Let \(\mathcal{I}=(\mathcal{I},\leq)\) be a filtered poset.
Let
$
(u,v) \colon f\to  g
$
in
$
\Mor
(\Fun
(\mathcal{I},\Chnn(\vect_k)))
$. Suppose that both
$
u$ and 
$v
$
satisfy kernel-cokernel stabilization. 
Then, for every \(q\in\NN\), there exist $\Delta(q) \in \ZZ$ and
\(A_q\in\mathcal{I}\) such that
$
\Delta(q) = \dim_k\Image (H_q(g_A))
-
\dim_k\Image (H_q(f_A))$  for all \(A\geq A_q\).
\end{theorem}

We will use this theorem to establish weak stabilization of the associated persistent Betti number functionals.

We first record some consequences of the exactness of filtered colimits.
For $C,D\in\Fun(\cI,\Ch(\Vect_k))$
and a morphism $f\colon C\to D$, we write
$
C_\infty
:=
\colim C,
\
f_\infty
:=
\colim f
\colon
C_\infty\to D_\infty.
$
Since filtered colimits of vector spaces are exact, they commute with
homology.
Hence, for every $q\in \ZZ$, there is a canonical isomorphism
$
\colim H_q(C)
\cong
H_q(C_\infty).
$ For a short exact sequence
$
0\to C\to C'\to Q\to 0
$ in
$\Fun
(\cI,\Ch(\Vect_k))$, we denote by $\delta_{q,A}\colon H_q(Q_A)\to H_{q-1}(C_A)
$ the connecting homomorphism at $A$, for each $A \in \cI$, and by
$
\delta_{q,\infty}\colon H_q(Q_\infty)\to H_{q-1}(C_\infty)
$ the connecting homomorphism associated with the colimit short exact sequence. Notice that we have the following canonical isomorphisms
\[
\colim_{A\in\cI}\Ker(\delta_{q,A})
\cong
\Ker(\delta_{q,\infty}) \text{ and } 
\colim_{A\in\cI}\Image (\delta_{q,A})
\cong
\Image  (\delta_{q,\infty}).\]

In what follows, we consider filtered diagrams of degreewise finite-dimensional chain complexes. Note, however, that the colimit of such a diagram need not be degreewise finite-dimensional in general.

We first prove the following lemmas.
\begin{lemma}\label{lem:stabi}
Let
$
0\to C
\xrightarrow{} C'
\xrightarrow{} Q
\to 0
$
be a short exact sequence in
$\Fun(\mathcal{I},\Ch(\vect_k))$, and let
$q\in\mathbb{Z}$.

\begin{enumerate}
\item
If $C \to C'$ satisfies cokernel stabilization, then there exists $A_{\mathrm{cok}}\in \cI$ such that, for every
$A\geq A_{\mathrm{cok}}$, the canonical morphism induces an isomorphism
$
\Ker(\delta_{q,A})
\cong
\Ker(\delta_{q,\infty})
$.

\item If $C' \to Q$ satisfies kernel stabilization, then there exists $A_{\mathrm{ker}}\in \cI$ such that, for every $A\geq A_{\mathrm{ker}}$, the canonical morphism induces an isomorphism
 $
\Image(\delta_{q,A})
\cong
\Image(\delta_{q,\infty})
$.
\end{enumerate}
\end{lemma}

\begin{proof}
We first prove (1). 
Since $C \to C'$ satisfies cokernel stabilization, there exists $A_q \in \cI$ such that 
$
H_q(Q_A)\to H_q(Q_\infty)
$
is an isomorphism for every $A\geq A_q$. This induces the following commutative diagram for any $A,B$ with $A_q \leq A \leq B$:
\[
\begin{tikzcd}[]
\Ker(\delta_{q,A})
  \arrow[r, hook]
  \arrow[d]
&
H_q(Q_A)
  \arrow[r, "\delta_{q,A}"]
  \arrow[d, "\cong"']
&
H_{q-1}(C_A)
  \arrow[d]
\\
\Ker(\delta_{q,B})
  \arrow[r, hook]
  \arrow[d]
&
H_q(Q_B)
  \arrow[r, "\delta_{q,B}"]
  \arrow[d, "\cong"']
&
H_{q-1}(C_B)
  \arrow[d]
\\
\Ker(\delta_{q,\infty})
  \arrow[r, hook]
&
H_q(Q_\infty)
  \arrow[r, "\delta_{q,\infty}"]
&
H_{q-1}(C_\infty).
\end{tikzcd}
\]
The induced maps between the kernels are injective.  In addition, we have
\[
\dim_k\Ker(\delta_{q,\infty})
\leq
\dim_k H_q(Q_\infty)
=
\dim_k H_q(Q_{A_q})
<
\infty,
\]
by assumption. This implies that there exists $A_{\mathrm{cok}}\geq A_q$ such that
$
\Ker(\delta_{q,A})
\to
\Ker(\delta_{q,\infty})
$
is an isomorphism for every $A\geq A_{\mathrm{cok}}$. 

The proof of (2) is similar to that of (1).
\end{proof}


The preceding lemma yields the following stabilization result.

\begin{lemma}\label{lem:stability}
\begin{enumerate}
\item In the setting of Lemma~\ref{lem:stabi}(1), there exist $\Delta(q) \in \ZZ$ and \(A_q \in\mathcal{I}\) such that
$\Delta(q) = \dim_k H_q(C'_A)-\dim_k H_q(C_A)
$  for all \(A\geq A_q\).
\item In the setting of Lemma~\ref{lem:stabi}(2),  there exist $\Delta(q) \in \ZZ$ and \(A_q\in\mathcal{I}\) such that 
$\Delta(q) = \dim_k H_q(C'_A)-\dim_k H_q(Q_A)
$  for all \(A\geq A_q\).
\end{enumerate}
\end{lemma}

\begin{proof}
We first prove~(1). For every \(A\in\mathcal{I}\),
Lemma~\ref{lem:dimension-identities-long-exact-sequence} gives
\[
\dim_k H_q(C'_A)-\dim_k H_q(C_A)
=
\dim_k\Ker(\delta_{q,A})
+\dim_k\Ker(\delta_{q+1,A})
-\dim_k H_{q+1}(Q_A).
\]
By Lemma~\ref{lem:stabi}(1),
the first two terms on the
right-hand side are constant for all sufficiently large \(A\).
 Moreover, cokernel stabilization implies that the map
$
  H_{q+1}(Q_A)\longrightarrow H_{q+1}(Q_\infty)
$
is an isomorphism for all sufficiently large \(A\). Hence
\(\dim_k H_{q+1}(Q_A)\) is also eventually constant.
Since \(\mathcal{I}\) is filtered, the finitely many stabilization
indices involved admit a common upper bound. Thus there exists
\(A_q \in\mathcal{I}\) such that
$
\dim_k H_q(C'_A)-\dim_k H_q(C_A)
$
is constant for \(A\geq A_q\).

We next prove~(2). For every \(A\in\mathcal{I}\),
Lemma~\ref{lem:dimension-identities-long-exact-sequence} gives
\[
\dim_k H_q(C'_A)-\dim_k H_q(Q_A)
=
\dim_k H_q(C_A)
-\dim_k\Image(\delta_{q,A})
-\dim_k\Image(\delta_{q+1,A}).
\]
By Lemma~\ref{lem:stabi}(2),
the last two terms on the right-hand side are eventually constant. Then, the desired assertion follows by an argument similar to that in~(1).
\end{proof}

The following proposition may be viewed as a common generalization of
the two statements in Lemma~\ref{lem:stability}. The morphism is no longer
required to be injective as in part~(1) or surjective as in part~(2).
Instead, it is assumed to satisfy kernel-cokernel stabilization.

\begin{proposition}
\label{prop:local-kernel-cokernel-stability}
Let
$
u\colon C\to D
$
be a morphism in
\(\Fun(\mathcal{I},\Ch(\vect_k))\).
Suppose that \(u\) satisfies kernel-cokernel stabilization.
Then,  for every \(q\in\mathbb{Z}\), there exist  $\Delta(q) \in \ZZ$ and \(A_q\in\mathcal{I}\) such that
$\Delta(q) = \dim_k H_q(D_A)-\dim_k H_q(C_A)$ for all \(A\geq A_q\).
\end{proposition}

\begin{proof}
Consider the image factorization of \(u\):
$
C
\overset{e}{\twoheadrightarrow}
\Image (u)
\overset{m}{\lhook\joinrel\to}
D.
$
Since \(u\) satisfies cokernel (resp. kernel) stabilization, \(m\) (resp., $e$) satisfies cokernel (resp. kernel) stabilization.
Applying Lemma~\ref{lem:stability}\textup{(1)} to the short exact sequence
\[
0
\to
\Image (u)
\xrightarrow{\,m\,}
D
\to
\Coker(m)
\to
0,
\]
we obtain \(A_m\in\mathcal{I}\) such that  $\dim_k H_q(D_A)
-
\dim_k H_q(\Image (u_A))$ is constant for $A \geq A_m$.
Similarly, applying Lemma~\ref{lem:stability}\textup{(2)} to the short exact sequence
\[
0
\to
\Ker(e)
\to
C
\xrightarrow{\,e\,}
\Image (u)
\to
0,
\]
we obtain \(A_e\in\mathcal{I}\) such that  $ \dim_k H_q(\Image (u_A))
-
\dim_k H_q(C_A)$
is constant for \(A\geq A_e\).
Since \(\mathcal{I}\) is filtered, there exists \(A_q\in\mathcal{I}\) such that
$
A_e\leq A_q
$ and $
A_m\leq A_q.
$ Therefore, for any $A \geq A_q$, 
\[
\begin{aligned}
\dim_k H_q(D_A)-\dim_k H_q(C_A)
={}&
(
\dim_k H_q(D_A)
-
\dim_k H_q(\Image (u_A))
)
\\
&+
(
\dim_k H_q(\Image (u_A))
-
\dim_k H_q(C_A)
)
\end{aligned}
\]
is constant.
\end{proof}

We next study the induced morphism between the mapping cones of a commutative square.

\begin{proposition}
\label{prop:cone-homology-difference-stability}
Let $(u,v)\colon f\to g
$ be a morphism 
in
$
\Mor
(
\Fun
(\mathcal{I},\Ch(\vect_k))
)
$. 
\begin{enumerate}
    \item 
If \(u\) and \(v\) satisfy cokernel  stabilization, then so does \(\Cone (u,v)\).
\item 
If \(u\) and \(v\) satisfy kernel  stabilization, then so does \(\Cone (u,v)\).
\item If \(u\) and \(v\) satisfy kernel-cokernel  stabilization, then for any $q \in \ZZ$, there exist  $\Delta(q) \in \ZZ$ and \(A_q\in\mathcal{I}\) such that 
$\Delta(q) = \dim_k H_q(\Cone(g_A))
-
\dim_k H_q(\Cone(f_A))$ for all \(A\geq A_q\). 
\end{enumerate}
\end{proposition}

\begin{proof}
We first prove (1) and (2). 
Let \(c:=\Cone (u,v)\). Let
$
\overline{g}\colon
\Coker(u)\to\Coker(v),\ \widehat{f}\colon
\Ker(u)\to\Ker(v)
$
be the morphisms induced by the universal property of cokernels and kernels, respectively.  By the
basic properties of mapping cones, there are natural isomorphisms
\begin{equation}\label{eq:cone-kernel-cokernel}
\Coker(c)
\cong
\Cone (\overline{g}),
\qquad
\Ker(c)
\cong
\Cone (\widehat{f});
\end{equation}
see Equation~\ref{eq:cone_cokernel}. Consequently, \(c\) satisfies cokernel (resp. kernel) stabilization if and only if \(\Cone (\overline{g})\) (resp. \(\Cone (\widehat{f})\)) stabilizes degreewise.

Suppose that
$
u\colon C\to C'$ and $
v\colon D\to D'
$
satisfy cokernel stabilization. 
For every $q \in \ZZ$, since \(\mathcal{I}\) is filtered, there
exists \(A_q\in\mathcal{I}\) such that, for every
\(A_q\leq A\leq B\), the following diagram is commutative:
\[
\begin{tikzcd}[column sep=large, row sep=large]
\Coker(u_{A,q})
  \arrow[r, "\sim"]
  \arrow[d, "\overline{g}_{A,q}"']
&
\Coker(u_{B,q})
  \arrow[d, "\overline{g}_{B,q}"]
\\
\Coker(v_{A,q})
  \arrow[r, "\sim"']
&
\Coker(v_{B,q}),
\end{tikzcd}
\]
where the horizontal maps are isomorphisms. 
Thus, by taking an upper bound $A_{\text{cone}, q}$ of $A_q$ and $A_{q-1}$, the induced morphism
$
\Cone (\overline{g}_{A})_q
\to
\Cone (\overline{g}_{B})_q
$
is an isomorphism for all $A_{\text{cone}, q}\leq A\leq B$. Thus,
\(\Cone (\overline{g})\) stabilizes degreewise.
This completes the proof of  (1). 
The statement (2) follows by the same argument as (1).

Finally, we prove (3).
Assume that  \(u\) and \(v\) satisfy kernel-cokernel stabilization. Then, by (1) and (2),  the morphism \(c\) satisfies kernel-cokernel stabilization. 
Thus, Proposition~\ref{prop:local-kernel-cokernel-stability}  yields the desired result.
\end{proof}

We now prove the main result of this section.

\begin{proof}[Proof of Theorem~\ref{cor:main_result}]
Let $q \in \NN$ and 
$
(u,v) \colon f\to  g
$
in
$
\Mor
(\Fun 
(\mathcal{I},\Chnn(\vect_k)))
$. 
For each \(m\in\{0,..., q\}\) and \(A\in\mathcal{I}\), set
$
\Delta_m (A)
:=
\dim_k\Image (H_m(g_A))
-
\dim_k\Image (H_m(f_A))
$.
By Lemma~\ref{lem:mapping-cone-rank-identity}, we have
\begin{align*}
\Delta_m(A)
={}&
(
\dim_k H_m(D'_A)-\dim_k H_m(D_A)
) \\
&+
(
\dim_k H_{m-1}(C'_A)-\dim_k H_{m-1}(C_A)
) \\
&-
(
\dim_k H_m(\Cone (g_A))
-
\dim_k H_m(\Cone (f_A))
) \\
&-\Delta_{m-1}(A).
\end{align*}
Using this recurrence, we prove by induction on \(m\) that
\(\Delta_m(A)\) is eventually constant in \(A\) for every
\(m=0,\ldots,q\).

For \(m=0\), since \(C\) and \(C'\) vanish in negative degrees, we have
$
H_{-1}(C_A)=H_{-1}(C'_A)=0.
$
Hence
\begin{align*}
\Delta_0(A)
={}&
(
\dim_k H_0(D'_A)-\dim_k H_0(D_A)
) \\
&-
(
\dim_k H_0(\Cone (g_A))
-
\dim_k H_0(\Cone (f_A))
).
\end{align*}
By Propositions~\ref{prop:local-kernel-cokernel-stability} and \ref{prop:cone-homology-difference-stability}, respectively, both differences are eventually constant. Hence, \(\Delta_0(A)\) is eventually constant.


Now let \(0\leq m<q\), and suppose that \(\Delta_m(A)\) is constant for
all sufficiently large \(A\). 
By Proposition~\ref{prop:local-kernel-cokernel-stability}, the differences
$
\dim_k H_{m+1}(D'_A)-\dim_k H_{m+1}(D_A)
$ and 
$
\dim_k H_m(C'_A)-\dim_k H_m(C_A)
$
are constant for all sufficiently large \(A\). Moreover, by
Proposition~\ref{prop:cone-homology-difference-stability},  
$
\dim_k H_{m+1}(\Cone (g_A))
-
\dim_k H_{m+1}(\Cone (f_A))
$
is constant for all sufficiently large \(A\). Since \(\mathcal{I}\) is filtered, there is a common index beyond which
these three differences and \(\Delta_m(A)\) are all constant. Applying
the recurrence in degree \(m+1\), we conclude that \(\Delta_{m+1}(A)\)
is constant for all sufficiently large \(A\).
This completes the proof.
\end{proof}

%% file: 3.5_addone.tex
\section{Criteria for stabilization and bounds of persistent Betti numbers}\label{sec:add}

In this section, we introduce \emph{add one maps} and define admissibility in terms of translation invariance,  \emph{add one kernel-cokernel stabilization}, and the \emph{add one kernel-cokernel boundedness condition}. We then show that the latter two conditions yield stabilization and bounds for the add one cost (Equation~\eqref{eq:add-one-cost}) of persistent Betti numbers.



\subsection{Add one maps and admissibility} \label{sec:add one-stabilization}
Let $d \in \NN_{\geq1}$. For $P\subseteq \mathbb{R}^d$, let $\mathcal{F}(P)$ be the poset of finite subsets of $P$, ordered by inclusion. It is filtered, and we regard it as a category in the usual way.
If $P\subseteq P'$, there is a canonical inclusion functor
$
\mathcal{F}(P)\hookrightarrow \mathcal{F}(P').$

We consider a functor 
\[ \mathcal{A}\colon \cU \to \Fun ( \mathbb{R}, \Chnn (\vect_k) ) 
.\] 
For \(A\subseteq\mathbb{R}^d\) and \(x\in\mathbb{R}^d\), set
$A+x:=\{a+x\mid a\in A\}$. We say that $\cA$ is \emph{translation invariant} if $\cA (F + x) \cong \cA (F)$ for every $F \in \cU$ and $x \in \RR^d$.

For \(F\in\cU\), we set \(F^+:=F\cup\{0\}\), where \(0\) denotes the
origin of \(\RR^d\). Thus, if $0 \in F$, then $F^+ =F$. Let
$ \operatorname{id}_{\cU}\colon \cU\to\cU, \ F\mapsto F, 
$
and 
$\operatorname{id}_{\cU}^+\colon \cU\to\cU, \ F\mapsto F^+:=F\cup\{0\}$. 
These functors induce a natural transformation \[ u^+\colon \mathcal{A} \to \mathcal{A}^+, \qquad \mathcal{A}^+ := \mathcal{A}\circ\operatorname{id}_{\cU}^+. \] For \(F\in\cU\) and \(t\in\mathbb{R}\), we denote its components by 
$u_{F}^+\colon \mathcal{A}(F) \to \mathcal{A}(F^+)
$
and 
$\add_{F,t}\colon \mathcal{A}(F)(t) \to \mathcal{A}(F^+)(t)$, respectively. 
 We call natural transformations of this form, and each of their
  components, \emph{add one maps}.
We then define the add one kernel-cokernel boundedness condition for \(\mathcal{A}\).

\begin{definition} \label{def:boundedness} We say that \(\mathcal{A}\) satisfies the \emph{add one kernel-cokernel boundedness condition} if there exist functions 
$\rho\colon\mathbb{R}\to[0,\infty)$, $
\tau\colon\mathbb{Z} \times \RR \to\NN
$, $c\colon\mathbb{Z}\times\mathbb{R}\to(0,\infty)$ such that, for every \(F\in\cU\), \(j\in\mathbb{Z}\), and \(t\in\mathbb{R}\), 
\[
\dim_k\Ker((\add_{F,t})_j) + \dim_k\Coker((\add_{F,t})_j) \leq c(j,t) \left( 1+ \#(F\cap B_0(\rho(t))) \right)^{\tau(j,t)}. \] \end{definition}

We next fix a subset \(P\subseteq\mathbb{R}^d\), not necessarily finite.
For the functors $\iota_P\colon \mathcal{F}(P)  \to \cU, \ F\mapsto F, $ and $i_P^+\colon \mathcal{F}(P)  \to \cU, \ F\mapsto F^+$, we have the functors
$ \cA_P := \mathcal{A}\circ\iota_P, \ \cA_P^+ := \mathcal{A}\circ i_P^+$ in $\Fun \left( \mathcal{F}(P) , \Fun ( \mathbb{R}, \Chnn (\Vect_k) ) \right)$.  
The inclusions \(F\subseteq F^+\) induce an add one map 
\[ \add_P:= u^+|_{\mathcal{F}(P) } \colon \mathcal{A}_P \to \mathcal{A}_P^+. \] 
For $F\in\mathcal{F}(P) $ and $t \in \mathbb{R}$, we denote its components by
  $ \add_{P,F}\colon \cA_P(F) = \mathcal{A}(F) \to \cA_P^+(F) = \mathcal{A}(F^+).
  $ 
  Via the natural equivalence
  \[ \Fun \left( \mathcal{F}(P) , \Fun ( \mathbb{R}, \Chnn (\vect_k) ) \right) \cong \Fun \left( \mathbb{R}, \Fun ( \mathcal{F}(P) , \Chnn (\vect_k) ) \right), \] we regard \(\cA_P\) and \(\cA_P^+\) as \(\mathbb{R}\)-indexed families of functors on \(\mathcal{F}(P) \). 
  More explicitly, for each \(t\in\mathbb{R}\), we have 
  $
  \mathcal{A}_{P,t}\colon \mathcal{F}(P)  \to \Chnn (\vect_k), \ F\mapsto\mathcal{A}(F)(t)$  and   
  $\mathcal{A}_{P,t}^+\colon \mathcal{F}(P)  \to \Chnn (\vect_k), \ F\mapsto\mathcal{A}(F^+)(t)$. 
Thus,  \(u_P^+\) induces an add one map 
\[
u_{P,t}^+\colon \cA_{P,t}\to\cA_{P,t}^+.
\]

We define the following conditions, which will be used to give a sufficient condition for persistent Betti number functionals to be weakly stabilizing.
\begin{definition}
\label{def:add one-kernel-cokernel-stabilization}
We say that \(\mathcal{A}\) satisfies \emph{add one kernel (resp. cokernel) stabilization at \(t \in \RR\)} if, for every locally finite subset  \(P\subseteq\mathbb{R}^d\), the add one map
$
\add_{P,t}\colon
\mathcal{A}_{P,t}
\to
\mathcal{A}_{P,t}^{+}
$
satisfies kernel (resp. cokernel) stabilization.
If \(\mathcal{A}\) satisfies both add one kernel stabilization and add one cokernel stabilization at \(t\), then we say that \(\mathcal{A}\) satisfies
\emph{add one kernel-cokernel stabilization at \(t\)}.

If any one of these stabilization conditions holds for every \(t\in\mathbb{R}\), we say that \(\mathcal{A}\) satisfies the corresponding pointwise stabilization condition.
\end{definition}

We now define admissibility.
\begin{definition}\label{def:admi}
For $t \in \RR,$ we say that \(\mathcal{A}\) is \emph{admissible at \(t\)} if it satisfies the following conditions:
\begin{enumerate}
\item
\(\mathcal{A}\) is translation invariant.
\item
\(\mathcal{A}\) satisfies add one kernel-cokernel stabilization at
\(t\).
\item
\(\mathcal{A}\) satisfies add one kernel-cokernel boundedness.
\end{enumerate}
If $\cA$ is admissible for every $t$, we say that $\cA$ is pointwise admissible.
\end{definition}

All functors considered in Section~\ref{sec:example} are pointwise admissible.


\begin{remark}
\label{rem:origin-already-present}
In both Definitions~\ref{def:boundedness}
and~\ref{def:add one-kernel-cokernel-stabilization}, it suffices to
consider point configurations that do not contain the origin.
Indeed, for the boundedness condition in
Definition~\ref{def:boundedness}, if \(0\in F\) for \(F\in\cU\), then
\(F^+=F\) and hence \(\add_{F,t}=\id_{\cA(F)(t)}\). Thus, both its kernel and
cokernel vanish. It therefore suffices to verify the boundedness
condition for \(F\in\cU\) such that \(0\notin F\).
For the stabilization condition in Definition~\ref{def:add one-kernel-cokernel-stabilization}, suppose that \(0\in P\). Taking \(F_0=\{0\}\), we have \(F^+=F\), and hence
\(\add_{P,t,F}=\id_{\cA(F)(t)}\), for every \(F\in\mathcal{F}(P) \) satisfying
\(F_0\subseteq F\). The required stabilization conditions therefore hold
automatically. Consequently, it suffices to verify the definition for any \(P\subseteq\mathbb{R}^d\) with \(0\notin P\).
\end{remark}

\subsection{Add one stabilization and bounds for persistent Betti numbers}
\label{subsec:deterministic-add one}
In this subsection, we show that add one kernel-cokernel stabilization
and boundedness imply, respectively, stabilization and boundedness for the add one costs of the associated persistent Betti number functionals. These results will be used in Subsection~\ref{subsec:CLT} to verify
the probabilistic assumptions required for central limit theorems.

To formulate these results, we first recall the add one cost. For a functional \(H\colon\cU\to\mathbb{R}\), its \emph{add one cost function} for $x \in \RR^d$ is
the map
\begin{equation}\label{eq:add-one-cost}
    D_xH\colon\cU  \rightarrow\mathbb{R},
\ F \mapsto
H(F \cup\{x\})-H(F).
\end{equation}
This measures the change in \(H\) caused by inserting a point at $x$. The functionals we have in mind are the persistent Betti number functionals \(\beta_q^{r,s}(\cA)\) defined in
Equation~\eqref{eq:PBN}.


\subsubsection{Add one stabilization for persistent Betti numbers}
\label{subsubsec:det-add-stab}
 We first prove the following. 
\begin{proposition}\label{prop:deterministic-weakly-stab}
Let
$
\cA \colon
\cU\rightarrow
\Fun 
(\RR,\Chnn (\vect_k))
$
be a functor, and fix \(q\in\NN\) and \(r,s\in\RR\)
with \(r\leq s\). 
Assume that \(\cA \) satisfies add one kernel-cokernel
stabilization at \(r\) and \(s\). Then, for any locally finite set \(P \subseteq \RR^d \),  there exist $\Delta(P) \in \ZZ$ and
\(F_0\in\mathcal{F}(P)  \) such that
$
\Delta(P) = D_0 \beta_q^{r,s}(\cA)(F)
$  for all \(F\supseteq F_0\) with $F \in \mathcal{F}(P) $.
\end{proposition}

\begin{proof}
  Let $P\subseteq \RR^d$ be locally finite. Evaluating the add one map $
\add_P\colon\cA _P\rightarrow\cA_P^+
$
at \(r\) and \(s\), we obtain the following commutative diagram in
\(\Fun 
(\mathcal{F}(P) ,\Chnn (\vect_k))\):
\[
\begin{tikzcd}
\cA _{P,r}
    \arrow[r,"u_{P,r}^+"]
    \arrow[d,"f"']
&
\cA _{P,r}^+
    \arrow[d,"g"]
\\
\cA _{P,s}
    \arrow[r,"u_{P,s}^+"']
&
\cA _{P,s}^+ ,
\end{tikzcd}
\]
where
$
f
$ and $
g
$
 are the structure morphisms. 
By assumption, \(u_{P,r}^+\) and \(u_{P,s}^+\) both satisfy kernel-cokernel stabilization. 
By Theorem~\ref{cor:main_result}, there
exists \(F_0\in\mathcal{F}(P) \) such that 
$
\dim_k\Image H_q(g_F)
-
\dim_k\Image H_q(f_F)
$
is constant for all \(F\supseteq F_0\) with $F \in \mathcal{F}(P) $. 
On the other hand, by definition, we have
$
\dim_k\Image H_q(f_F)
=
\beta_q^{r,s}(\cA (F))$ and $ 
\dim_k\Image H_q(g_F)
=
\beta_q^{r,s}
(\cA (F^+))$. Thus, the difference is equal to some
\(\Delta(P)\in\ZZ\) for every \(F\supseteq F_0\).
\end{proof}

\subsubsection{Add one boundedness for persistent Betti numbers}
\label{subsubsec:add one-bound}

We next discuss the consequence of add one kernel-cokernel boundedness.

We begin with two elementary lemmas from linear algebra.

\begin{lemma}
\label{lem:mapsdim}
Let
$
U \xrightarrow{\alpha}V\xrightarrow{\beta}W
$
be linear maps between finite-dimensional vector spaces. Then
\[
\dim_{k}\Coker(\beta\circ\alpha)
\leq
\dim_{k}\Coker(\alpha)
+
\dim_{k}\Coker(\beta)
\]
and
\[
\dim_{k}\Ker(\beta\circ\alpha)
\leq
\dim_{k}\Ker(\alpha)
+
\dim_{k}\Ker(\beta).
\]
\end{lemma}

\begin{proof}
There are exact sequences
\[
\Coker(\alpha)
\rightarrow
\Coker(\beta\circ\alpha)
\rightarrow
\Coker(\beta)
\rightarrow
0
\]
and
\[
0
\rightarrow
\Ker(\alpha)
\rightarrow
\Ker(\beta\circ\alpha)
\rightarrow
\Ker(\beta).
\]
 The
desired inequalities follow by taking dimensions in these exact sequences.
\end{proof}

\begin{lemma}
\label{lem:kercoker}
Let
\[
\begin{tikzcd}
V \arrow[r,"f"] \arrow[d,"u"']
  & V' \arrow[d,"u'"] \\
W \arrow[r,"g"']
  & W'
\end{tikzcd}
\]
be a commutative diagram of finite-dimensional vector spaces. Then
\[
-\dim_{k}\Ker(g)
\leq
\dim_{k}\Image (u')
-
\dim_{k}\Image (u)
\leq
\dim_{k}\Coker(f).
\]
\end{lemma}

\begin{proof}
We first prove the left-hand inequality. Let
$
\iota\colon \Image (u)\hookrightarrow W
$
denote the canonical inclusion, and let
$
\delta\colon \Image (u)
\twoheadrightarrow
\Image (g\circ u)
$
be the map induced by \(g\). By the universal properties of kernels and
images, we obtain the following commutative diagram, whose middle row is
exact:
\[
\begin{tikzcd}[column sep=3.5em, row sep=2.8em]
  & V
      \arrow[r,"u"]
      \arrow[dr,two heads]
  & W
      \arrow[r,"g"]
  & W'
  \\
  0
      \arrow[r]
  & \Ker(\delta)
      \arrow[r,hook]
  & \Image (u)
      \arrow[r,two heads,"\delta"]
      \arrow[u,hook,"\iota"]
  & \Image (g\circ u)
      \arrow[r]
      \arrow[u,hook]
  & 0
  \\
  & \Ker(g\circ\iota)
      \arrow[u,"\cong"]
      \arrow[ur,hook]
\end{tikzcd}
\]
It follows immediately from this diagram that
\[
\begin{aligned}
\dim_{k}\Image (u)
-
\dim_{k}\Image (g\circ u)
=
\dim_{k}\Ker(\delta) 
=
\dim_{k}\Ker(g\circ\iota) 
\leq
\dim_{k}\Ker(g).
\end{aligned}
\]
Moreover, the commutativity of the original diagram gives
\[\dim_{k} 
\Image (g\circ u)
= \dim_{k}
\Image (u'\circ f)
\leq \dim_{k}
\Image (u').
\]
Therefore, we have
$
\dim_{k}\Image (u)
-
\dim_{k}\Image (u')
\leq
\dim_{k}\Image (u)
-
\dim_{k}\Image (g\circ u) 
\leq
\dim_{k}\Ker(g)
$, and obtain the left-hand inequality.
The right-hand inequality follows by the dual argument.
\end{proof}

For \(x\in\mathbb{R}^d\) and \(R \geq0 \), let
$B_x(R)$ denote the closed ball of radius \(R\) centered at \(x\). For a set $A$, we write $\# A$ for its cardinality. Now we prove the following.
\begin{proposition}
\label{prop:local-boundedness} Let
$
\cA \colon
\cU\rightarrow
\Fun 
(\RR,\Ch (\vect_k))
$
be a functor. 
Let \(q\in\ZZ \) and \(r,s\in\RR\) with \(r\leq s\).  If \(\cA \) satisfies the add one kernel-cokernel boundedness condition, then there exist constants $R, C \geq0 $ and $m \in \NN$ such that 
\[
\left|
D_0\beta_q^{r,s}(\cA )(P\cap A)
\right|
\leq
C\left(
1+\# P\cap {B_0(R)} 
\right)^m
\]
for every locally finite set \(P\subseteq\RR^d\) and every bounded set
\(A\subseteq\RR^d\).
\end{proposition}

\begin{proof}
Let $P \subseteq \RR^d$ be locally finite. Take a bounded set $A\subseteq \RR^d$ and set $F:= P \cap A$. Notice that $F \in \cU$ since $P$ is locally finite.
Applying Lemma~\ref{lem:kercoker} to the
commutative diagram
\[
\begin{tikzcd}[column sep=2.5cm]
H_q(\cA (F)(r))
    \arrow[r,"H_q(u^+_{F,r})"]
    \arrow[d]
&
H_q(\cA (F^{+})(r))
    \arrow[d]
\\
H_q(\cA (F)(s))
    \arrow[r,"H_q(u^+_{F,s})"']
&
H_q(\cA (F^{+})(s))
\end{tikzcd}
\]
gives
\begin{equation}
\label{eq:first}
\left|
D_0\beta_q^{r,s}(\cA )(F)
\right|
\leq
\dim_{k}\Coker H_q(u^+_{F,r})
+
\dim_{k}\Ker H_q(u^+_{F,s}).
\end{equation}

For \(t\in\{r,s\}\), consider the image factorization
\begin{equation}
\label{eq:imdecomp}
\cA (F)(t)
\overset{\alpha_t}{\twoheadrightarrow}
\Image (u^+_{F,t})
\overset{\beta_t}{\hookrightarrow}
\cA (F^{+})(t),
\qquad
\beta_t\circ\alpha_t=u^+_{F,t}.
\end{equation}
Taking homology yields
\[
H_q(\cA (F)(t))
\xrightarrow{H_q(\alpha_t)}
H_q(\Image (u^+_{F,t}))
\xrightarrow{H_q(\beta_t)}
H_q(\cA (F^{+})(t)).
\]
Lemma~\ref{lem:mapsdim} therefore implies that
\begin{align}
\dim_{k}\Coker H_q(u^+_{F,t})
&\leq
\dim_{k}\Coker H_q(\alpha_t)
+
\dim_{k}\Coker H_q(\beta_t),
\label{eq:second-a}
\\
\dim_{k}\Ker H_q(u^+_{F,t})
&\leq
\dim_{k}\Ker H_q(\alpha_t)
+
\dim_{k}\Ker H_q(\beta_t).
\label{eq:second-b}
\end{align}

The image factorization~\eqref{eq:imdecomp} gives the two short exact
sequences
\[
0
\rightarrow
\Ker(u^+_{F,t})
\rightarrow
\cA (F)(t)
\xrightarrow{\alpha_t}
\Image (u^+_{F,t})
\rightarrow
0
\]
and
\[
0
\rightarrow
\Image (u^+_{F,t})
\xrightarrow{\beta_t}
\cA (F^{+})(t)
\rightarrow
\Coker(u^+_{F,t})
\rightarrow
0.
\]
Their associated long exact sequences in homology yield
\begin{align}
\dim_{k}\Ker H_q(\alpha_t)
&\leq
\dim_{k}(\Ker(u^+_{F,t}))_q,
\label{eq:third-a}
\\
\dim_{k}\Coker H_q(\alpha_t)
&\leq
\dim_{k}(\Ker(u^+_{F,t}))_{q-1},
\label{eq:third-b}
\\
\dim_{k}\Ker H_q(\beta_t)
&\leq
\dim_{k}(\Coker(u^+_{F,t}))_{q+1},
\label{eq:third-c}
\\
\dim_{k}\Coker H_q(\beta_t)
&\leq
\dim_{k}(\Coker(u^+_{F,t}))_q.
\label{eq:third-d}
\end{align}
Combining \eqref{eq:first}--\eqref{eq:third-d}, we obtain
\[
\begin{aligned}
\left|
D_0\beta_q^{r,s}(\cA )(F)
\right|
&\leq
\dim_{k}(\Ker(u^+_{F,r}))_{q-1}  + \dim_{k}(\Coker(u^+_{F,r}))_q \\ 
&+
\dim_{k}(\Ker(u^+_{F,s}))_q +
\dim_{k}(\Coker(u^+_{F,s}))_{q+1}.
\end{aligned}
\]

Define $R,C\geq 0$ and $m \in \NN$, independently of $P$ and \(A\), by 
\[
R:=\max\{\rho(r),\rho(s)\},
\
m:=\max_{\substack{j \in \{q-1,q,q+1\}\\
                t \in \{r,s\}}}
\tau(j,t) ,\ C:= 4
\max_{\substack{j \in \{q-1,q,q+1\}\\
                t \in \{r,s\}}}c(j,t)
\]
 where $\rho, \tau$ and $c$ are the maps given by the add one kernel-cokernel boundedness assumption. By the boundedness condition again we obtain
\[
\left|
D_0\beta_q^{r,s}(\cA )(F)
\right|
\leq
C\left(
1+\#(F\cap B_0(R))
\right)^m \leq  C\left(
1+\#(P\cap B_0(R))
\right)^m .
\]
Since $R,m,C$ are independent of $P$ and $A$, we obtain the desired assertion.
\end{proof}

%% file: 5_Criterion.tex
\section{Central limit theorems for admissible functors}\label{sec:criterion}

In this section, we derive central limit theorems for persistent Betti
number functionals associated with admissible functors.

\subsection{Probabilistic preliminaries}

This subsection recalls the stabilization and moment conditions used in
the central limit theorems for Poisson and Gibbs point processes.

Let \(d\in\mathbb{N}_{\geq1}\). 
We write
$
(\mathbb{R}^d,\cB(\mathbb{R}^d),\mathfrak{m})
$ for the \(d\)-dimensional Lebesgue measure space, where
\(\cB(\mathbb{R}^d)\) denotes the Borel \(\sigma\)-algebra and
\(\mathfrak{m}\) denotes Lebesgue measure. 
Let
$\lf$ denote the space of all locally finite subsets of $\mathbb{R}^{d}$, equipped with the $\sigma$-algebra generated by the maps
$
    P\mapsto\#(P\cap B),
$
where $B\subseteq\mathbb{R}^{d}$ is a bounded Borel set. Via the natural inclusion $\cU \to \lf$, we equip $\cU$ with the induced $\sigma$-algebra. 


\subsubsection{Poisson point processes}
\label{subsec:poisson-preliminaries}
We refer the reader to \cite{LP17} for background on Poisson point
processes. Let \(\eta\) be a homogeneous Poisson point process with unit intensity on $\RR^d$ defined on a probability space $(\Omega, \mathrm{F}, \mathbb{P})$. Since \(\eta\) is simple, we identify it with its support 
$
\mathcal{P}(\omega)
:=
\left\{
x\in\mathbb{R}^d:
\eta(\omega,\{x\})>0
\right\}.
$ For a Borel set
\(A\subseteq\mathbb{R}^d\), set
$
\mathcal{P}_A:=\mathcal{P}\cap A
$. Recall that, if \(A\in\cB(\mathbb{R}^d)\) and
\(\mathfrak{m}(A)<\infty\), then the expected value
$
\mathbb{E}\!\left[(\#\mathcal{P}_A)^m\right]
$ is finite for every \(m\in\mathbb{N}\).

For $A \subseteq \RR^d$, we set 
$
A^{(r)}
:=
\{
x\in\mathbb{R}^d
\mathrel \mid
\inf_{a\in A}\lVert x-a\rVert\leq r
\}
$ and, we denote $\partial A$ by the boundary of $A$.
Throughout this paper, let
\(\{W_n\}_{n\in\mathbb{N}}\) be a sequence of Borel subsets of
\(\mathbb{R}^d\) satisfying the following conditions:
\begin{enumerate}
\item[\textnormal{(A1)}]
For every \(n\in\mathbb{N}\),
$
\mathfrak{m}(W_n)=n.
$
\item[\textnormal{(A2)}] 
The sequence $\{W_n\}$ tends to $\RR^d$, i.e.,
$
\bigcup_{n\geq 1}\bigcap_{m\geq n}W_m
=
\mathbb{R}^d.
$

\item[\textnormal{(A3)}]
For every \(r>0\),
$
\lim_{n\to\infty} n^{-1}
{\mathfrak{m}((\partial W_n)^{(r)})}
=
0.
$

\item[\textnormal{(A4)}]
There exists a constant \(\gamma>0\) such that
$
\operatorname{diam}(W_n)
\leq
\gamma n^\gamma
$
for every \(n\in\mathbb{N}\).
\end{enumerate}
For example, let
\begin{equation}\label{eq:Deltan}
\Lambda_n
:=
[-n^{1/d}/2,n^{1/d}/2)^d.    
\end{equation}
 Then \(\Lambda_n\) is a half-open cube of side length \(n^{1/d}\), and the
sequence \(\{\Lambda_n\}_{n\in\mathbb{N}}\) satisfies
\textnormal{(A1)}--\textnormal{(A4)}.  We set
$
\mathcal{W}
:=
\{
W_n+x
\bigm|
n\in\mathbb{N},\ x\in\mathbb{R}^d
\bigr\}
$.

A functional
$H\colon\cU \rightarrow\mathbb{R}
$
is said to be \emph{translation invariant} if
$
H(F+y)=H(F)
$
for every \(F\in\cU \) and \(y\in\mathbb{R}^d\).
We say that \(H\) satisfies the \emph{Poisson bounded moment condition}
on \(\mathcal{W}\) if
\[
\sup_{\substack{A\in\mathcal{W}, 0\in A}}
\mathbb{E}
\left[
\left|D_0H(\cP_A)\right|^4
\right]
<
\infty.
\]
We say that \(H\) is \emph{weakly stabilizing} on \(\mathcal{W}\) with respect to $\mathcal{P}$ if there exists a random variable
\(D\colon\Omega\to\mathbb{R}\) such that, for every sequence
\(\{A_n\}_{n\in\mathbb{N}}\) with \(A_n\in\mathcal{W}\) tending to
$
\mathbb{R}^d,
$
one has
$
D_0H(\cP_{A_n})
\xrightarrow{\text{a.s}}
D$
 as 
$n\to\infty.
$

We conclude this subsection by recalling a sufficient condition for a central
limit theorem for functionals of finite point configurations.

\begin{lemma}[{\cite[Theorem~3.1]{PY01}}]
\label{lem:clt}
Let
$
H\colon\cU  \rightarrow\mathbb{R}
$
be a measurable translation invariant functional. Suppose that \(H\) is
weakly stabilizing on \(\mathcal{W}\) and satisfies the Poisson bounded moment condition on \(\mathcal{W}\).
Then there exists \(\sigma^2 \in [0, \infty) \) such that
$n^{-1} \Var(H (\cP_{W_n}) ) \to \sigma^2$ and 
$n^{-1/2}( H(\cP_{W_n})  - \EE[ H(\cP_{W_n})  ] ) \xrightarrow{d} \mathcal{N}(0,\sigma^2)$  as $n \to \infty$.
\end{lemma}

\subsubsection{Gibbs point processes}
\label{subsec:gibbs-preliminaries}
We next recall the central limit theorem for Gibbs point processes (\cite[Theorem~3]{HOS25}). We use this result to show that our algebraic method applies beyond the Poisson setting. Since our
aim is to verify the assumptions of this theorem for persistent Betti number functionals, we recall only the terminology and assumptions needed for this application and refer the reader
to~\cite{HOS25, Der19} for further background on Gibbs point processes.

Let $\mathcal{X}$ be an infinite-volume Gibbs point process with
Papangelou conditional intensity
$
    \kappa\colon
    \mathbb{R}^{d}\times
    \mathcal{F}_{\mathrm{lf}}(\mathbb{R}^{d})
    \rightarrow[0,\infty).
$
We assume that $\kappa$ is translation invariant, bounded by some
$\alpha _0>0$, and has finite interaction range
$r_0<r_c(\alpha _0)$, where $r_c(\alpha _0)$ denotes the infimum of all $r > 0$ such that $ \bigcup_{x \in \cP_{\alpha_0} } B_x (r/2)$ has an unbounded connected component with positive probability, where $\mathcal
{P}_{\alpha_0}$ is a Poisson point process with intensity $\alpha_0$.
Following \cite{HOS25}, we assume that
\begin{equation}
\label{eq:gibbs-process-assumption}
0\leq\kappa(x,P)\leq\alpha_{0}
\quad\text{and}\quad
r_{0}<r_{c}(\alpha_{0}).
\tag{G}
\end{equation}
Recall that a homogeneous Poisson point process of intensity $\alpha_{0}$ corresponds to the constant Papangelou
conditional intensity $\kappa= \alpha_{0}$.

For $n\in \ZZ_{\geq0}, z \in \RR^d$, set
$
    Q_{n}:=[-n/2,n/2]^{d}
$ and $ Q_{z,n}:=  Q_{n} + z$.
We set $
\mathcal{Q}
:=
\{
Q_n+z
\bigm|
n\in\mathbb{N},\ z\in\mathbb{Z}^d
\bigr\}
$.
For a functional $H\colon\mathcal{F}(\mathbb{R}^{d})\to\mathbb{R}$, define
$H_{n}(P):=H(P\cap Q_{n})$ for $P\in\lf.$ 
We say that $H$ satisfies the Gibbs moment condition if there exist constants $c,
p>0$ such that
\begin{equation}
\label{eq:gibbs-moment-condition}
\mathbb{E}\left[
\left|
H_{n}(\mathcal{X})-
H_{n}(\mathcal{X}\setminus (Q_{m}+z))
\right|^{5}
\right]
\leq
c m^{p}
\end{equation}
for every $n,m\in\mathbb{N}$ and $z\in\mathbb{Z}^{d}$.
We say that $H$ is weakly stabilizing with respect to $\mathcal{X}$
if, almost surely,
\begin{equation}
\label{eq:gibbs-weak-stabilization}
\lim_{n\to\infty}
(
H(\mathcal{X}\cap A_n)
-
H(
(\mathcal{X}\cap A_n )\setminus Q_{l}
)
)
\end{equation}
exists for every $l \in \ZZ_{\geq0}$ and every sequence $\{A_n \in \mathcal{Q}\}_n$ tending to $\RR^d$.

\begin{lemma}[{\cite[Theorem~3]{HOS25}}]
\label{lem:gibbs-clt}
Let $\mathcal{X}$ be an infinite-volume Gibbs point process with
translation-invariant Papangelou conditional intensity satisfying
\eqref{eq:gibbs-process-assumption}. Let
$H\colon\mathcal{F}(\mathbb{R}^{d})\to\mathbb{R}$ be a measurable
translation invariant functional satisfying
 \eqref{eq:gibbs-moment-condition} and
\eqref{eq:gibbs-weak-stabilization}. Then there exists
$\sigma^{2}\in[0,\infty)$ such that
$
    \mathfrak{m}(Q_{n})^{-1}
    \operatorname{Var}(H_{n}(\mathcal{X}))
    \rightarrow \sigma^{2}
$
and
$
\mathfrak{m}(Q_{n})^{-1/2}
\left(
 H_{n}(\mathcal{X})
 -
 \mathbb{E}[H_{n}(\mathcal{X})]
\right)
\xrightarrow{d}
N(0,\sigma^{2})
$
as $n\to\infty$.
\end{lemma}



\subsection{Central limit theorems for persistent Betti number functionals}\label{subsec:CLT}

We now prove the main theorems in this paper.

\begin{theorem}\label{thm:clt_main}
Let
$
\cA \colon
\cU\rightarrow
\Fun 
(\RR,\Chnn (\vect_k))
$
be a functor. Let \(q\in\NN\) and \(r,s\in\RR\) with \(r\leq s\).
Suppose that $\cA$ is admissible at $r$ and $s$.
Assume further that the 
functional
$
\beta_q^{r,s}(\cA )
$
is measurable.

Let \(\mathcal{P}\) be a homogeneous Poisson point process of unit
intensity on \(\RR^d\).
For every sequence \(\{W_n\}_{n\geq 1}\) satisfying
conditions \textup{(A1)}--\textup{(A4)}, there exists $
\sigma_{q,r,s}^2\in[0,\infty)
$ such that as $n \to \infty$, 
$ n^{-1}
\Var\!\left[
\beta_q^{r,s}
(\cA) (\mathcal{P}_{W_n}))
\right]
\to
\sigma_{q,r,s}^2
$
and 
$
n^{-1/2}
(
\beta_q^{r,s}
(\cA )(\mathcal{P}_{W_n}))
-
\mathbb{E}\!\left[
\beta_q^{r,s}
(\cA) (\mathcal{P}_{W_n}))
\right]
)
\xrightarrow{\mathrm{d}}
\mathcal{N}(0,\sigma_{q,r,s}^2).
$
\end{theorem}


\begin{proof}
By Lemma~\ref{lem:clt}, it suffices to show that the functional \(\beta_q^{r,s}(\cA )\) is translation invariant and weakly stabilizing and satisfies the Poisson bounded moment condition.

The translation invariance of \(\cA \) implies that the functional
\(\beta_q^{r,s}(\cA )\) is translation invariant. 

We next show that \(\beta_q^{r,s}(\cA )\) is weakly stabilizing. 
 Let $P:= \cP(\omega)$ for $\omega \in \Omega$. By Proposition~\ref{prop:deterministic-weakly-stab}, there exist a constant $D(\omega) \in \ZZ$ and
\(F_0\in\mathcal{F}(P)  \) such that $D(\omega) = D_0 \beta_q^{r,s}(\cA(F))$ for all \(F\supseteq F_0\) with $F \in \mathcal{F}(P) $.
Let \( \{A_n \in \cW \}_{n\geq 1}\) be any sequence tending to $\RR^d$.
Since \(F_0\) is finite, there exists \(N\in\mathbb{N}\) such that
\(F_0\subseteq A_n\) for all \(n\geq N\). 
If $P$ is locally finite, then
$P_{A_n} \in \cU$. Thus,
$
(D_0\beta_q^{r,s}(\cA ))(P_{A_n})
=
D(\omega)
$
for all 
$
n\geq N.
$
Since $\cP$ is almost surely locally finite, we can obtain
\[
\lim_{n\to\infty}
(D_0\beta_q^{r,s}(\cA ))(P_{A_n})= D \text{ a.s.}
\]
for a random variable $D \colon \Omega \to \RR.$  Thus, the functional is weakly stabilizing.

Finally, we show that
\(\beta_q^{r,s}(\cA )\)  satisfies the Poisson bounded moment condition.
Let \(A\in\mathcal{W}\) with \(0\in A\). Since $\cP$ is almost surely locally finite, 
 by Proposition~\ref{prop:local-boundedness}, there exist constants $R,C\geq0$ and $m \in \NN$, independent of $A$,  such that 
\[
\left|
D_0\beta_q^{r,s}(\cA )(\mathcal{P}_A)
\right|
\leq
C(
1+\# \mathcal{P}_{B_0(R)} 
)^m \text{ a.s.}
\]
Since \(\#\mathcal{P}_{B_0(R)}\) has moments of all orders, the functional satisfies the Poisson bounded moment condition.
\end{proof}

Notice that Theorem~\ref{thm:clt-main} is derived from this theorem by taking $W_n = \Lambda_{n}$ given by Equation~\ref{eq:Deltan}.

We next consider Gibbs point processes. In contrast to the Poisson setting,
the observation windows in this subsection are the cubes $Q_{n}$
introduced in Subsection~\ref{subsec:gibbs-preliminaries}.

\begin{theorem}
\label{thm:gibbs-clt-main}
Let $
\cA\colon
\mathcal{F}(\mathbb{R}^{d})
\rightarrow
\Fun(
 \mathbb{R},
 \Chnn(\vect_{k}))
$ be a functor. Let $q\in\mathbb{N}$ and
$r,s\in\mathbb{R}$ with $r\leq s$. Suppose that $\cA$ is
admissible at $r$ and $s$ and that
$\beta_{q}^{r,s}(\cA)$ is measurable.

Let $\mathcal{X}$ be an infinite-volume Gibbs point process with
translation-invariant Papangelou conditional intensity satisfying
\eqref{eq:gibbs-process-assumption}. Then there exists
$\sigma_{q,r,s}^{2}\in[0,\infty)$ such that
$
 \mathfrak{m}(Q_{n})^{-1}
\operatorname{Var}\left[
 \beta_{q}^{r,s}(\cA)(\mathcal{X}_{Q_{n}})
\right]
\rightarrow
\sigma_{q,r,s}^{2}
$
and
$
 \mathfrak{m}(Q_{n})^{-1/2}
\left(
 \beta_{q}^{r,s}(\cA)(\mathcal{X}_{Q_{n}})
 -
 \mathbb{E}\left[
  \beta_{q}^{r,s}(\cA)(\mathcal{X}_{ Q_{n}})
 \right]
\right)
\xrightarrow{d}
\mathcal{N}(0,\sigma_{q,r,s}^{2})
$
as $n\to\infty$.
\end{theorem}

\begin{proof}
By Lemma~\ref{lem:gibbs-clt}, it suffices to show that the functional \(\beta_q^{r,s}(\cA )\) is translation invariant and weakly stabilizing with respect to $\mathcal{X}$ \eqref{eq:gibbs-weak-stabilization} and satisfies the Gibbs moment condition \eqref{eq:gibbs-moment-condition}.

The translation invariance of $\cA$ implies that $\beta_{q}^{r,s}(\cA)$ is translation invariant.

We next show that the functional is weakly stabilizing with respect to
$\mathcal{X}$. 
Let $P\subseteq \RR^d$ be locally finite. Let
$\{A_n \in \mathcal{Q}\}_{n\geq 1}$ be any sequence tending to $\RR^d$.
By Proposition~\ref{prop:deterministic-weakly-stab}, there exist
$\Delta(P)\in\mathbb{Z}$ and $F_{0}\in \mathcal{F}(P) $ such
that
$
    D_{0}\beta_{q}^{r,s}(\cA)(F)= \Delta(P)
$
whenever $F_{0}\subseteq F\subseteq P$ and $F$ is finite.
Since $F_{0}$ is finite, there exists $N\in\mathbb{N}$ such that
$
    F_{0}\subseteq A_n
$
for every $n\geq N$. Consequently,
$
    D_{0}\beta_{q}^{r,s}(\cA)(P_{A_n})= \Delta(P)
$
for every $n \geq N$. In particular,
$
    \lim_{n\to\infty}
    D_{0}\beta_{q}^{r,s}(\cA)(P_{A_n})
$
exists for every locally finite $P$. Thus, the condition \cite[Equation~(12)]{HOS25} holds. Then, \cite[Lemma~34]{HOS25} implies
that the functional is weakly stabilizing.

Finally, we show that $\beta_{q}^{r,s}(\cA)$  satisfies the Gibbs moment condition.  By Proposition~\ref{prop:local-boundedness}, there exist constants $R,C\geq 0$ and $m\in\mathbb{Z}_{\geq 0}$ such that
$
    |D_{0}\beta_{q}^{r,s}(\cA)(F)|
    \leq
    C(1+\#(F\cap B_{0}(R)))^{m}
$
for every $F \in \cU$.  Thus, for every $z\in\mathbb{R}^{d}$,
$
    D_{z}\beta_{q}^{r,s}(\cA)(F)=D_{0}\beta_{q}^{r,s}(\cA)(F-z),
$
and therefore
$
    |D_{z}\beta_{q}^{r,s}(\cA)(F)|
    \leq
    C\bigl(1+\#(F\cap B_{z}(R))\bigr)^{m}.
$
After increasing the constant if necessary, we obtain constants
$c,R>0$ such that
$$
    |D_{y}\beta_{q}^{r,s}(\cA)(F)|
    \leq
    c
    \exp\left(
      c
      \#(F\cap B_{y}(R))
    \right).
$$
Thus, 
\cite[Equation~(11)]{HOS25} holds. It follows from
\cite[Lemma~33]{HOS25} that $\beta_{q}^{r,s}(\cA)$ satisfies the Gibbs moment condition.
\end{proof}

%% file: 6_Examples.tex
\section{Applications to  persistent Betti numbers}\label{sec:example}

In this section, we apply Theorem~\ref{thm:clt_main} to derive central limit theorems for persistent Betti number functionals. By Theorem~\ref{thm:gibbs-clt-main}, the same results hold for Gibbs point processes whose Papangelou intensities satisfy \eqref{eq:gibbs-process-assumption}, when the observation windows are the cubes \(Q_n\).

We first recover the central limit theorem for persistent Betti numbers arising from simplicial complex filtrations established in \cite[Theorem~5.2]{HST18}. In the case of the Čech filtration, the corresponding Gibbs-process result recovers \cite[Corollary~6]{HOS25}. We then prove a central limit theorem for persistent Betti numbers arising from \(\ell^p\)-Vietoris-Rips simplicial homology. Finally, we prove central limit theorems for persistent Betti numbers associated with blurred and persistent magnitude homology of Poisson random geometric graphs. As a special case, we obtain a central limit theorem for the magnitude Betti numbers of Poisson random geometric graphs.

All persistent Betti number functionals considered below are measurable, and we omit separate verifications of this fact.



\subsection{Persistent homology of filtrations}
\label{subsec:persistent-homology-filtrations}

In this subsection, we recover from our framework the central limit
theorem for persistent Betti numbers established in
\cite[Theorem~5.2]{HST18}. We first recall the setting.

 Let
$
\kappa\colon\cU\rightarrow[0,\infty]
$
be a measurable map, and let
$
\rho\colon[0,\infty]\rightarrow[0,\infty]
$
be an increasing function. We call \(\kappa\) a \emph{filtration function} if it satisfies the following conditions:
\begin{enumerate}
    \item[\textnormal{(K1)}] For every \(\sigma,\tau\in\cU\), if
    \(\sigma\subseteq\tau\), then
    $
    0\leq\kappa(\sigma)\leq\kappa(\tau).
    $

    \item[\textnormal{(K2)}] The function \(\kappa\) is translation invariant; that is,
    $
    \kappa(\sigma+x)=\kappa(\sigma)
    $
    for every \(\sigma\in\cU\) and \(x\in\mathbb{R}^d\).

    \item[\textnormal{(K3)}] If \(t<\infty\), then \(\rho(t)<\infty\), and
    $
    \lVert x-y\rVert
    \leq
    \rho(\kappa(\{x,y\}))
    $
    for every \(x,y\in\mathbb{R}^d\).
\end{enumerate}
We further set
$
\kappa(\emptyset):=0
$ and \(\kappa(\{x\})=0\) for every \(x\in\mathbb R^d\).

For \(F\in\cU\) and
\(t\in[0,\infty]\), define
$
K_\kappa(F)(t)
:=
\left\{
\sigma\subseteq F
\;\middle|\;
\kappa(\sigma)\leq t
\right\}.
$
Condition \textnormal{(K1)} ensures that \(K_\kappa(F)(t)\) is a simplicial complex. Moreover,
$
K_\kappa(F)(t)\subseteq K_\kappa(F)(t')
$
whenever
$
t\leq t',
$
and hence \(K_\kappa(F)\) defines a filtration of simplicial complexes, called the \emph{\(\kappa\)-filtration}. This class includes the filtrations of \v{C}ech and Vietoris-Rips complexes (\cite[Example 1.3]{HST18}).
Condition \textnormal{(K2)} implies that
$
\kappa(\{x\})=\kappa(\{y\})
$
for every \(x,y\in\mathbb{R}^d\). In what follows, we set
$
\kappa(\{x\})=0
$
for every
$
x\in\mathbb{R}^d.
$

We note the following immediate consequence of conditions \textnormal{(K1)} and \textnormal{(K3)}. Let \(F\in\cU\), \(t\in[0,\infty)\), and \(\sigma\in K_\kappa(F)(t)\). If \(0\in\sigma\), then, for every \(x\in\sigma\),
$
\lVert x\rVert
\leq
\rho(\kappa(\{0,x\}))
\leq
\rho(\kappa(\sigma))
\leq
\rho(t).
$
Consequently,
\begin{equation}
\label{eq:inball}
0\in\sigma
\ \text{and}\
\sigma\in K_\kappa(F)(t)
\quad\Longrightarrow\quad
\sigma\subseteq B_0(\rho(t)).
\end{equation}

We extend the associated filtered chain complex to an
\(\mathbb{R}\)-indexed filtered chain complex by declaring it to be the zero complex for \(t<0\). We continue to use the notation \(K_\kappa(F)\) for this extension. This defines a functor
\[
\mathcal{A}_\kappa\colon
\cU
\rightarrow
\Fun
(
\mathbb{R},
\Chnn(\mathsf{vect}_k)
),
\qquad
F\mapsto C(K_\kappa(F)).
\]

\begin{proposition}
\label{prop:PH}
The functor \(\mathcal{A}_\kappa\) is pointwise admissible.
\end{proposition}

\begin{proof}
Condition \textnormal{(K2)} implies that
\(\mathcal{A}_\kappa\) is translation invariant.

We next prove add one kernel-cokernel stabilization. For \(t<0\),
the chain complex \(\mathcal{A}_\kappa(F)(t)\) is the zero complex, so
the assertion is trivial. Let \(t\geq 0\), and fix  \(P\subseteq\mathbb{R}^d\). We can assume that $0 \notin P$ by Remark~\ref{rem:origin-already-present}.
For every \(F\in\mathcal{F}(P) \), the add one 
map
$
u_{P,F,t}^+\colon
\mathcal{A}_\kappa(F)(t)
\rightarrow
\mathcal{A}_\kappa(F^+)(t)
$
is injective. Hence, \(u_{P,t}^+\) satisfies kernel stabilization.
For each \(q\in\NN\), we have
\[
\Coker(u_{P,F,t}^+)_q
\cong
\left\langle
\left\{
\tau\subseteq F^+
\;\middle|\;
0\in\tau,\ 
\#\tau=q+1,\ 
\kappa(\tau)\leq t
\right\}
\right\rangle_k.
\]
By~\eqref{eq:inball}, every simplex generating the right-hand side is
contained in \(B_0(\rho(t))\). Set
$
F_0:=P\cap B_0(\rho(t)).
$
 For every
\(F,F'\in\mathcal{F}(P) \) satisfying
$
F_0\subseteq F\subseteq F',
$
the induced map
$
\Coker(u_{P,F,t}^+)_q
\rightarrow
\Coker(u_{P,F',t}^+)_q
$
is an isomorphism because their generators coincide.
Thus, \(u_{P,t}^+\) satisfies cokernel stabilization.

Finally, we prove add one kernel-cokernel boundedness. Let
\(0 \notin F\in\cU\), \(q\in\NNn\), and \(t\in\mathbb{R}\). 
The
degree-\(q\) component of the add one map
$
(\add_{F,t})_q\colon
C(K_\kappa(F)(t))_q
\rightarrow
C(K_\kappa(F^+)(t))_q
$
is injective, and hence
$
\dim_k\Ker((\add_{F,t})_q)=0.
$
If \(t\geq0\), then~\eqref{eq:inball} gives
\[
\begin{aligned}
\dim_k\Coker
((\add_{F,t})_q)
&=
\#\left\{
\tau\in K_\kappa(F^+)(t)
\;\middle|\;
0\in\tau,\ \#\tau=q+1,\ \kappa(\tau) \leq t,\ \tau \subseteq F^+ 
\right\}
\\
&\leq
\left(
1+
\#(F\cap B_0(\rho(t)))
\right)^{q+1}.
\end{aligned}
\]
For \(t<0\), the corresponding kernel and cokernel vanish.
Therefore, \(\mathcal{A}_\kappa\) satisfies add one kernel-cokernel
boundedness.
\end{proof}

The following result recovers \cite[Theorem~5.2]{HST18}.

\begin{corollary}
Fix \(q\in\NN\) and \(0\leq r\leq s<\infty\).
Let \(\{W_n\}_{n\in\mathbb{N}}\) be a sequence satisfying conditions \textnormal{(A1)}--\textnormal{(A4)}, and let \(\mathcal{P}\) be a homogeneous Poisson point process of unit intensity on \(\mathbb{R}^d\). Then there exists a constant
$
\sigma^2\in[0,\infty)
$
such that, as $n \to \infty$,
$
n^{-1/2}(
\beta_q^{r,s}(\mathcal{A}_\kappa)(\mathcal{P}_{W_n})
-
\mathbb{E}\!\left[\beta_q^{r,s}(\mathcal{A}_\kappa)(\mathcal{P}_{W_n})\right])
\xrightarrow{d}
\mathcal{N}(0,\sigma^2).
$
\end{corollary}

\begin{proof}
The assertion follows from Proposition~\ref{prop:PH} and Theorem~\ref{thm:clt_main}.
\end{proof}

\subsection{\(\ell_p\)-Vietoris-Rips}

In this subsection, we consider persistent Betti number
functionals associated with the \(\ell_p\)-Vietoris-Rips simplicial
sets studied in \cite{Sim20,IX26}. For \(p=1\), this construction yields
blurred magnitude homology \cite[Section~2.3]{IX26}, whereas for
\(p=\infty\), it gives persistent homology isomorphic to
the standard Vietoris-Rips persistent homology
\cite[Theorem~18(2)(b)]{Sim20}.

\subsubsection{\texorpdfstring{$\ell_p$}{lp}-Vietoris-Rips simplicial homology}
We recall $\ell_p$-Vietoris-Rips simplicial homology; see \cite{IX26} for details.
Let $(X,d_X)$ be an extended metric space, and let
$p\in[1,\infty]$. For $q\in \NN$ and a tuple
$
\sigma=(x_0,\ldots,x_q)\in X^{q+1},
$
 its \emph{$\ell_p$-weight} $w_{X,p}=w_p$ is given by
\begin{equation}
\label{eq:lp-weight}
w_p(\sigma)
=
\max_{0\leq i_0<\cdots<i_m\leq q}
\left\|
(
d_X(x_{i_0},x_{i_1}),
\ldots,
d_X(x_{i_{m-1}},x_{i_m})
)
\right\|_p,
\end{equation}
where $\lVert-\rVert_p$ denotes the
$\ell_p$-norm and maximum is taken over all
strictly increasing subsequences $i_0 < \cdots < i_m$ of $0,\ldots, q$. Here,
\[
w_\infty(\sigma)
=
\operatorname{diam}\{x_0,\ldots,x_q\},
\qquad
w_1(\sigma)
=
d_X(x_0,x_1)+\cdots+d_X(x_{q-1},x_q).
\]

For $r\geq 0$, define the \emph{$\ell_p$-Vietoris-Rips simplicial set}
$\operatorname{VR}_{\leq r}^p(X)$ by
\[
\operatorname{VR}_{\leq r}^p(X)_q
:=
\left\{
(x_0,\ldots,x_q)\in X^{q+1}
\;\middle|\;
w_p(x_0,\ldots,x_q)\leq r
\right\}.
\]
The face maps are given by deleting entries, and the degeneracy maps
are given by repeating entries.
Set
$
C_{\leq r}^p(X)
:=
N(\operatorname{VR}_{\leq r}^p(X)),
$
where $N(S)$ denotes the normalized chain complex of a simplicial set. In particular,
$
C_{q,\leq r}^p(X)
=
N_q(\operatorname{VR}_{\leq r}^p(X))
$
is the $k$-vector space generated by $q$-simplices
$(x_0,\ldots,x_q)$ satisfying
$
w_p(x_0,\ldots,x_q)\leq r
$ with
$
x_i\neq x_{i+1}$ 
for every $0\leq i<q.
$
The boundary operator
$
\partial_q\colon
C_{q,\leq r}^p(X)
\rightarrow
C_{q-1,\leq r}^p(X)
$
is given by
\[
\partial_q(x_0,\ldots,x_q)
=
\sum_{i=0}^q
(-1)^i
(x_0,\ldots,\widehat{x_i},\ldots,x_q).
\]

For $r\leq s$, we have
$
\operatorname{VR}_{\leq r}^p(X)
\subseteq
\operatorname{VR}_{\leq s}^p(X).
$
Extending $C_{\leq r}^p(X)$ by the zero complex for $r<0$, we therefore
obtain a functor
$
C_{\leq -}^p(X)
\in
\Fun
(
\mathbb{R},
\Chnn(\Vect_k)
).
$
If $X$ is finite, then each chain group is finite-dimensional.
Replacing $\leq$ by $<$ in the above construction yields
$C_{<-}^p(X)$. For each $r\in\mathbb{R}$, define
$
C_r^p(X)
:=
\Coker
(
C_{<r}^p(X)
\rightarrow
C_{\leq r}^p(X)
).
$
This induces the canonical isomorphism

\begin{equation}\label{eq:magnitude-construction}
C_-^p(X) \cong 
\Coker
(
C_{<-}^p(X)
\rightarrow
C_{\leq -}^p(X)
).    
\end{equation}

Let $\sigma=(x_0,\ldots,x_q)$ be a $q$-simplex of
$\operatorname{VR}_{\leq r}^p(X)$. Then we have
$
d_X(x_i,x_j)\leq r
$ for $
0\leq i,j\leq q.
$
Indeed,
$
d_X(x_i,x_j)
\leq
\operatorname{diam}\{x_0,\ldots,x_q\}
=
w_\infty(\sigma)
\leq
w_p(\sigma)
\leq
r.
$
Consequently, for every $0\leq i\leq q$, we have
\begin{equation}
\label{eq:lp}
\{x_0,\ldots,x_q\}
\subseteq
B_{x_i}(r).
\end{equation}

\subsubsection{CLTs for \texorpdfstring{$\ell_p$}{lp}-Vietoris-Rips simplicial homology}

We regard each \(F\in\cU\) as a finite metric space equipped with the metric induced by the Euclidean metric on \(\mathbb{R}^d\). By setting the corresponding chain complex to be the zero complex at every negative parameter, we obtain a functor
\[
\mathcal{A}^p
\colon
\cU
\rightarrow
\Fun
(
\mathbb{R},
\Chnn(\vect_k)
),
\qquad
F
\mapsto
C_{\leq -}^p(F).
\]

\begin{proposition}
\label{prop:lp}
Fix \(p\in[1,\infty]\). Then \(\mathcal{A}^p\) is pointwise admissible.
\end{proposition}

\begin{proof}
The translation invariance of \(\mathcal{A}^p\) follows from that of the Euclidean metric.
 Let \(t\geq 0\). By Equation~\eqref{eq:lp}, every vertex of a simplex in
\(\operatorname{VR}_{\leq t}^p(F^+)\) containing the origin lies in
\(B_0(t)\). Therefore, the same argument as in
Proposition~\ref{prop:PH} shows that \(\mathcal{A}^p\) satisfies pointwise add one kernel-cokernel stabilization and add one kernel-cokernel boundedness.
\end{proof}

\begin{corollary}
\label{cor:clt-lp}
Fix \(q\in\NN\), \(0\leq r\leq s\), and
\(p\in[1,\infty]\).
Let \(\{W_n\}_{n\in\mathbb{N}}\) be a sequence satisfying conditions
\textnormal{(A1)}--\textnormal{(A4)}, and let \(\mathcal{P}\) be a homogeneous Poisson point process of unit intensity on \(\mathbb{R}^d\). Then there exists a constant
$
\sigma^2\in[0,\infty)
$
such that as $n \to \infty$, 
$ n^{-1}
\Var [
\beta_q^{r,s}(\mathcal{A}^p) (\mathcal{P}_{W_n})
]
\to
\sigma^2
$
and 
$
n^{-1/2}
(
 \beta_q^{r,s}(\mathcal{A}^p)(\mathcal{P}_{W_n})
-
\mathbb{E} [\beta_q^{r,s}(\mathcal{A}^p) (\mathcal{P}_{W_n})]
)
\xrightarrow{d}
\mathcal{N}(0,\sigma^2).
$
\end{corollary}

\begin{proof}
The assertion follows from Proposition~\ref{prop:lp} and  Theorem~\ref{thm:clt_main}.
\end{proof}


\subsection{\texorpdfstring{$\ell_p$}{lp}-Vietoris-Rips simplicial sets of Poisson random geometric graphs}
\label{subsec:lp}

In this subsection, we establish a central limit theorem for persistent Betti numbers arising from the \(\ell_p\)-Vietoris-Rips homology of Poisson random geometric graphs.

Fix \(p\in[1,\infty]\). Let \(\mathsf{FinMet}_1\) denote the category
of finite extended metric spaces and \(1\)-Lipschitz maps. A \(1\)-Lipschitz map
$
f\colon (X,d_X)\rightarrow(Y,d_Y)
$
induces a morphism of \(\mathbb{R}\)-indexed chain complexes
$
C_{\leq -}^p(f)
\colon
C_{\leq -}^p(X)
\rightarrow
C_{\leq -}^p(Y).
$
In degree \(q\), it sends a simplex
\((x_0,\ldots,x_q)\) to
$
(f(x_0),\ldots,f(x_q))$.
Consequently, we obtain a functor
\[
C_{\leq -}^p
\colon
\mathsf{FinMet}_1
\rightarrow
\Fun
(
\mathbb{R},
\Chnn(\vect_k)
).
\]

Let \(F\in\cU\). For \(t\in\mathbb{R}\), let \(G(F,t)\) be the graph with vertex set \(F\), in which two distinct vertices \(x,y\in F\) are joined by an edge if and only if
$
\lVert x-y\rVert\leq t.
$
In particular, \(G(F,t)\) has no edges when \(t<0\). We equip \(G(F,t)\) with the shortest-path metric in which every edge has length \(1\), and define the distance between vertices belonging to distinct connected components to be \(\infty\). In this way, \(G(F,t)\) becomes a finite extended metric space.

For \(r\leq s\), the graph inclusion
$
G(F,r)\hookrightarrow G(F,s)
$
induces a \(1\)-Lipschitz map of extended metric spaces. Indeed, adding edges cannot increase shortest-path distances. We therefore obtain a functor
$
G(F,-)
\colon
\mathbb{R}
\rightarrow
\mathsf{FinMet}_1.
$
Moreover, for an inclusion \(F\subseteq F'\), the induced map
$
G(F,t)\rightarrow G(F',t)
$
is \(1\)-Lipschitz. Consequently, \(G\) defines a functor
\[
G
\colon
\cU\times\mathbb{R}
\rightarrow
\mathsf{FinMet}_1.
\]
We regard the first parameter as the scale parameter of the \(\ell_p\)-Vietoris-Rips simplicial set and the second as the connection radius of the graph. Composing the above functors yields a functor
$
\cM^p
\colon
\cU
\rightarrow
\Fun
(
\mathbb{R}^2,
\Chnn(\vect_k)
)
$
defined by
$
\cM^p(F)(a,t)
:=
C_{\leq a}^p(G(F,t)).
$
In particular, for each fixed \(t\in\mathbb{R}\), we obtain the two functors
\begin{equation}
\label{eq:lp-graph-functor2}
\cM_{(t,-)}^p
\colon
\cU
\rightarrow
\Fun
(
\mathbb{R},
\Chnn(\vect_k)
),
\qquad
F
\mapsto
C_{\leq t}^p(G(F,-))
\end{equation}
and
\begin{equation}
\label{eq:lp-graph-functor}
\cM_{(-,t)}^p
\colon
\cU
\rightarrow
\Fun
(
\mathbb{R},
\Chnn(\vect_k)
),
\qquad
F
\mapsto
C_{\leq -}^p(G(F,t)).
\end{equation}
These functors satisfy the following.
\begin{proposition}
\label{prop:lp-bmh}
For every \(t\geq 0\), the functors
$
\cM_{(t,-)}^p
$ and $
\cM_{(-,t)}^p
$
are pointwise admissible.
\end{proposition}

\begin{proof}
Since the Euclidean metric is invariant under translations, both
\(\cM_{(t,-)}^p\) and \(\cM_{(-,t)}^p\) are translation invariant. By Proposition~\ref{prop:2para-equiv}, 
it suffices to prove the stabilization and boundedness assertions for
\(\cM_{(-,t)}^p\). We therefore set
$
\cM:=\cM_{(-,t)}^p.
$

We first prove pointwise add one kernel-cokernel stabilization. The assertion is immediate for \(r<0\), so assume that \(r\geq 0\). Let
\(P\subseteq\mathbb{R}^d\) be locally finite, and consider the add one map,
$
u_{P,r}^+
\colon
\cM_{P,r}
\rightarrow
\cM_{P,r}^+.
$
For every \(F\in\mathcal{F}(P) \), the map
$
u_{P,r,F}^+
\colon
C_{\leq r}^p(G(F,t))
\rightarrow
C_{\leq r}^p(G(F^+,t))
$
is injective. Hence \(u_{P,r}^+\) satisfies kernel stabilization.
We next prove cokernel stabilization. Let \(E\in\mathcal{F}(P) \) and \(q\in\NN\), and let
$
\sigma=(x_0,\ldots,x_q)
$
be a simplex representing a basis element of
\(\Coker(u_{P,r,E}^+)_q\). Then either \(\sigma\) contains the origin, or its \(\ell_p\)-weight is strictly greater than \(r\) in \(G(E,t)\) but at most \(r\) in \(G(E^+,t)\).
In the former case, the distance in \(G(E^+,t)\) from each vertex of \(\sigma\) to the origin is at most \(r\). In the latter case, the graph distance between some pair of vertices of \(\sigma\) is strictly decreased by adding the origin. In particular, a new shortest path between such a pair must pass through the origin, and hence the distance in \(G(E^+,t)\) from every vertex of \(\sigma\) to the origin is at most \(2r\). Since every edge has Euclidean length at most \(t\), in either case we have
\begin{equation}
\label{eq:lp-graph-inball}
\{x_0,\ldots,x_q\}
\subseteq
B_0(2tr).
\end{equation}
Set
$
F_0:=P\cap B_0(3tr),
$
which is finite by the local finiteness of \(P\). Every simplex representing a basis element of the cokernel has all its vertices in
\(B_0(2tr)\), and every graph path of length at most \(r\) starting at one of these vertices is contained in
\(B_0(3tr)\). Thus, both before and after adding the origin, whether the \(\ell_p\)-weight of such a simplex is at most \(r\) is determined entirely by \(F_0\).
Consequently, for any
$
F_0\subseteq F'\subseteq F''
$ with $
F',F''\in\mathcal{F}(P) ,
$
the chain map
$
\Coker(u_{P,r,F'}^+)
\rightarrow
\Coker(u_{P,r,F''}^+)
$
is an isomorphism in every degree and hence an isomorphism of chain complexes. Therefore, \(u_{P,r}^+\) satisfies cokernel stabilization. It follows that \(\cM\) satisfies pointwise add one kernel-cokernel stabilization.

Finally, using Equation~\eqref{eq:lp-graph-inball}, the same counting argument for generators as in Proposition~\ref{prop:PH} shows that \(\cM\) satisfies add one kernel-cokernel boundedness.
\end{proof}

\begin{corollary}
\label{cor:clt-lp-graph}
Fix \(t\geq 0\), \(q\in\NN\), \(0\leq r\leq s\), and
\(p\in[1,\infty]\). Let
$
\cM
=
\cM_{(-,t)}^p
$ or $
\cM
=
\cM_{(t,-)}^p.
$
Let \(\{W_n\}_{n\in\mathbb{N}}\) be a sequence satisfying conditions
\textup{(A1)}--\textup{(A4)}, and let \(\mathcal{P}\) be a homogeneous Poisson point process of unit intensity on \(\mathbb{R}^d\). Then there exists a constant
$
\sigma^2\in[0,\infty)
$
such that as $n \to \infty$, 
$ n^{-1}
\Var\!\left[
\beta_q^{r,s}(\cM) (\mathcal{P}_{W_n})
\right]
\to
\sigma^2
$
and 
$
n^{-1/2}
\left(
 \beta_q^{r,s}(\cM)(\mathcal{P}_{W_n})
-
\mathbb{E} [\beta_q^{r,s}(\mathcal{M}) (\mathcal{P}_{W_n})]
\right)
\xrightarrow{d}
\mathcal{N}(0,\sigma^2).
$
\end{corollary}

\begin{proof}
The assertion follows from Proposition~\ref{prop:lp-bmh} and Theorem~\ref{thm:clt_main}.
\end{proof}

We next consider relative construction.

\begin{corollary}
\label{cor:relative_graph}
Fix \(p\in[1,\infty]\) and \(0\leq t_1\leq t_2\). Define
\[
\mathcal{R}_{t_1,t_2}^p
:=
\Coker
\left(\cM_{(t_1,-)}^p
\rightarrow \cM_{(t_2,-)}^p
\right)
\colon
\cU
\rightarrow
\Fun
(
\mathbb{R},
\Chnn(\vect_k)
).
\]
Then \(\mathcal{R}_{t_1,t_2}^p\) is pointwise admissible.
\end{corollary}

\begin{proof}
Translation invariance of $\cM_{(t_1,-)}^p$ and $\cM_{(t_2,-)}^p$ is immediate from that of  
$\cM^p$. By Proposition~\ref{prop:lp-bmh} and Proposition~\ref{prop:relative}(3), the functor
\(\mathcal{R}_{t_1,t_2}^p\) is pointwise admissible.
\end{proof}

\begin{corollary}
\label{cor:clt_PMG_graph}
Fix \(p\in[1,\infty]\), \(q\in\NN\),
\(0\leq r\leq s\), and \(0\leq t_1\leq t_2\). 
Let \(\{W_n\}_{n\in\mathbb{N}}\) be a sequence satisfying conditions
\textnormal{(A1)}--\textnormal{(A4)}, and let \(\mathcal{P}\) be a homogeneous Poisson point process of unit intensity on \(\mathbb{R}^d\). Then there exists a constant
$
\sigma^2\in[0,\infty)
$
such that as $n \to \infty$, 
$ n^{-1}
\Var\!\left[
\beta_q^{r,s}(\mathcal{R}_{t_1,t_2}^p)(\mathcal{P}_{W_n})
\right]
\to
\sigma^2
$
and 
$
n^{-1/2}
\left(
 \beta_q^{r,s}(\mathcal{R}_{t_1,t_2}^p)(\mathcal{P}_{W_n})
-
\mathbb{E} [\beta_q^{r,s}(\mathcal{R}_{t_1,t_2}^p) (\mathcal{P}_{W_n})]
\right)
\xrightarrow{d}
\mathcal{N}(0,\sigma^2).
$
\end{corollary}

\begin{proof}
The assertion follows from Corollary~\ref{cor:relative_graph} and Theorem~\ref{thm:clt_main}.
\end{proof}

\begin{remark}\label{rem:clt-PMG-graph}
Let \(m\in\mathbb{N}_{>0}\). Since the graph metric is integer-valued,
so is the \(\ell_1\)-weight. Therefore,
$
C_{<m}^1(G(-,-))
=
C_{\leq m-1}^1(G(-,-)).
$
Consequently, by Equation~\eqref{eq:magnitude-construction}, we have
\begin{equation}\label{eq:persistent-magnitude}
\cR_{m-1,m}^1
=
\Coker\left(
C_{\leq m-1}^1(G(-,-))
\rightarrow
C_{\leq m}^1(G(-,-))
\right) \cong C_m^1(G(-,-)).
\end{equation}
Therefore, the functor \(\cR_{m-1,m}^1\) 
provides a persistent version of magnitude homology.
Hence, Corollary~\ref{cor:clt_PMG_graph} yields a central limit theorem
for the persistent magnitude Betti numbers of Poisson random geometric
graphs. In particular, the corresponding central limit theorem for magnitude
Betti numbers follows by specialization.
\end{remark}

\begin{remark}
The authors of \cite{BFLW26} also consider a persistent version of magnitude homology on the category
\(\mathsf{FinMet}_{\mathrm{iso}}\) of finite metric spaces and
isometric embeddings.
Recall that our functor is defined by
$
C_-^1 \colon
\mathsf{FinMet}_1
\rightarrow
\Fun
\left(
\mathbb{R},
\Chnn(\vect_k)
\right).
$
The restriction of \(C_-^1\) to the subcategory
\(\mathsf{FinMet}_{\mathrm{iso}}\) agrees with the persistent version of magnitude
homology construction considered in \cite{BFLW26}.
\end{remark}

%% file: 9_Appendix.tex
\appendix \label{appendix}

\section{Admissible functors for images, kernels, and cokernels}
\label{subsec:closure}

We provide a method for constructing new admissible functors from given ones. 

Let \(\mathcal{I}:=(\mathcal{I},\leq)\) be a filtered poset.
We begin with the following lemma. It is analogous to Proposition~\ref{prop:cone-homology-difference-stability} for mapping cones.

\begin{lemma}
\label{lem:image_cokernel_kernel-local-cokernel-stability}
Let
$
(u,v)\colon f\to g
$
be a morphism in
$
\Mor
\left(
\Fun
(\mathcal{I},\Chnn(\vect_k))
\right)
$, given by the following commutative diagram:
\[
\begin{tikzcd}
C \arrow[r,"u"] \arrow[d,"f"'] & C' \arrow[d,"g"] \\
D \arrow[r,"v"'] & D'.
\end{tikzcd}
\]
Consider the induced morphisms
$w\colon\Image (f)\to\Image (g)$,
$c\colon\Coker (f)\to\Coker (g)
$, and 
$
e\colon\Ker (f)\to\Ker (g)$.
Then the following statements hold.
\begin{enumerate}
    \item If \(u\) (resp. $v$) satisfies cokernel (resp. kernel) stabilization, then so does \(w\).

    \item If \(u\) and \(v\) satisfy cokernel stabilization, then \(c\) satisfies cokernel stabilization. Moreover, if \(v\) satisfies kernel stabilization, then \(c\) satisfies kernel stabilization.

    \item If \(u\) and \(v\) satisfy kernel stabilization, then \(e\) satisfies kernel stabilization. Moreover, if \(u\) satisfies cokernel stabilization, then \(e\) satisfies cokernel stabilization.
\end{enumerate}
\end{lemma}

\begin{proof}
Fix \(q\in\NN\).
We first prove (1). Suppose that \(u\) satisfies cokernel stabilization. Then
there exists \(A_q\in\mathcal{I}\) such that, for every
\(A_q\leq A\leq B\), the morphism 
$
\Coker (u_{A})_q
\to
\Coker (u_{B})_q
$
is an isomorphism. It follows that the morphism
$
\Coker (w_{A})_q
\to
\Coker (w_{B})_q
$
is surjective. Hence the dimensions of the vector spaces
\(\Coker (w_{A})_q \) are non-increasing for \(A\geq A_q\).
Since these vector spaces are finite-dimensional, the morphisms above are eventually isomorphisms, and hence \(w\) satisfies cokernel stabilization. 
If $v$ satisfies kernel stabilization, then we can show that  \(w\) satisfies kernel stabilization by a similar argument.

We next prove (2). In general, there is a canonical natural isomorphism
\[
\Coker (c)
\cong
\Coker 
\left(
\Coker (u)
\to
\Coker (v)
\right).
\]
If \(u\) and \(v\) satisfy cokernel stabilization, this isomorphism shows
that \(\Coker (c)\) stabilizes degreewise. Thus \(c\) satisfies cokernel stabilization. Suppose, in addition, that \(v\) satisfies kernel stabilization. We have a
commutative diagram with exact rows
\[
\begin{tikzcd}
0 \arrow[r]
&
\Image (f) \arrow[r] \arrow[d,"w"']
&
D \arrow[r] \arrow[d,"v"]
&
\Coker (f) \arrow[r] \arrow[d,"c"]
&
0
\\
0 \arrow[r]
&
\Image (g) \arrow[r]
&
D' \arrow[r]
&
\Coker (g) \arrow[r]
&
0.
\end{tikzcd}
\]
The snake lemma yields an exact sequence
\[
\Ker (w)
\to
\Ker (v)
\to
\Ker (c)
\to
\Coker (w)
\to
\Coker (v).
\]
By the assumptions and (1), the four terms other than
\(\Ker (c)\) stabilize degreewise. For all sufficiently
large \(A\leq B\), applying the five lemma to the morphism between the
corresponding exact sequences at \(A\) and \(B\) shows that
$
\Ker (c_{A})_q
\to
\Ker (c_{B})_q
$
is an isomorphism. Therefore, \(c\) satisfies kernel stabilization. This
proves (2).

Finally, (3) follows by an argument dual to that in (2)
\end{proof}

We next consider a functor
$
\mathcal{M}\colon
\cU
\to
\Fun
(
\mathbb{R}^2,
\Chnn (\vect_k)
).
$
For each \(a\in\mathbb{R}\), this gives rise to two functors
\[
\mathcal{M}_{a,-}\colon
\cU
\to
\Fun
(
\mathbb{R},
\Chnn (\vect_k)
),
\qquad
F\mapsto\mathcal{M}(F)(a,-),
\]
and
\[
\cM_{-,a}\colon
\cU
\to
\Fun
(
\mathbb{R},
\Chnn (\vect_k)
),
\qquad
F\mapsto\mathcal{M}(F)(-,a).
\]
The following proposition shows that the two parameter directions play symmetric roles with respect to add one kernel-cokernel stabilization and add one kernel-cokernel boundedness.

\begin{proposition}
\label{prop:2para-equiv}
We have the following.
\begin{enumerate}
    \item For every \(a,b\in\mathbb{R}\), the functor
    \(\cM_{a,-}\) satisfies add one kernel-cokernel
    stabilization at \(b\) if and only if the functor
    \(\cM_{-,b}\) satisfies add one kernel-cokernel
    stabilization at \(a\).

    \item The functor \(\cM_{a,-}\) satisfies add one
    kernel-cokernel boundedness for every \(a\in\mathbb{R}\) if and
    only if the functor \(\cM_{-,b}\) satisfies add one
    kernel-cokernel boundedness for every \(b\in\mathbb{R}\).
\end{enumerate}
\end{proposition}

\begin{proof}
Both statements follow by interchanging the roles of the two coordinates
\(a\) and \(b\).
\end{proof}

Consequently, when verifying add one kernel-cokernel stabilization and
add one kernel-cokernel boundedness, it suffices to consider only one
of the two types of slices.

We next consider the natural transformation between two slices of a two-parameter functor.

\begin{proposition}\label{prop:relative}
Fix \(t_1\leq t_2\) and consider the natural transformation
$
\cM_{t_1,-}
\to
\cM_{t_2,-}
$. 
\begin{enumerate}
    \item If \(\cM_{t_1,-}\) and \(\cM_{t_2,-}\)
satisfy add one kernel-cokernel stabilization at \(r\), then 
$
\Image 
\left(
\cM_{t_1,-}
\to
\cM_{t_2,-}
\right)$, $\Coker
\left(
\cM_{t_1,-}
\to
\cM_{t_2,-}
\right),
$
and
$
\Ker
\left(
\cM_{t_1,-}
\to
\cM_{t_2,-}
\right)
$
satisfy add one kernel-cokernel stabilization at \(r\).

\item If $
\cM_{t_1,-}$ and $ \cM_{t_2,-}
$
satisfy add one kernel-cokernel boundedness, then 
$
\Image 
\left(
\cM_{t_1,-}
\to
\cM_{t_2,-}
\right)$, $
\ \Coker
\left(
\cM_{t_1,-}
\to
\cM_{t_2,-}
\right),
$
and
$
\Ker
\left(
\cM_{t_1,-}
\to
\cM_{t_2,-}
\right)
$
satisfy add one kernel-cokernel boundedness.

\item If $
\cM_{t_1,-}$ and $ \cM_{t_2,-}
$ are admissible at $r$, then then 
$
\Image 
\left(
\cM_{t_1,-}
\to
\cM_{t_2,-}
\right)$, $
\ \Coker
\left(
\cM_{t_1,-}
\to
\cM_{t_2,-}
\right),
$
and
$
\Ker
\left(
\cM_{t_1,-}
\to
\cM_{t_2,-}
\right)
$ are admissible at $r$.
\end{enumerate}
\end{proposition}

\begin{proof}
    We first prove (1). 
Let \(P\subseteq\mathbb{R}^d\) be locally finite. We obtain a commutative diagram in
\(\Fun(\mathcal{F}(P) ,\Chnn (\vect_k))\):
\[
\begin{tikzcd}
\cM_{P,t_1,r}
    \arrow[r,"\add_{P,t_1,r}"]
    \arrow[d,""']
&
\cM_{P,t_1,r}^{+}
    \arrow[d,""]
\\
\cM_{P,t_2,r}
    \arrow[r,"\add_{P,t_2,r}"']
&
\cM_{P,t_2,r}^{+},
\end{tikzcd}
\]
where vertical morphisms are the structure  morphisms. The desired assertion follows immediately from  Lemma~\ref{lem:image_cokernel_kernel-local-cokernel-stability}.

We next prove (2). 
Set
$
\mathcal{W}
:=
\Image 
\left(
\cM_{t_1,-}
\to
\cM_{t_2,-}
\right),
\
\mathcal{R}
:=
\Coker
\left(
\cM_{t_1,-}
\to
\cM_{t_2,-}
\right),
$
and
$
\mathcal{K}
:=
\Ker
\left(
\cM_{t_1,-}
\to
\cM_{t_2,-}
\right).
$
Fix \(q\in\NN\), \(F\in\cU\), and \(r\in\mathbb{R}\). Let
$
u^+\colon
\cM(F)(t_1,r)_q
\to
\cM(F^+)(t_1,r)_q
$
and
$
v^+\colon
\cM(F)(t_2,r)_q
\to
\cM(F^+)(t_2,r)_q,
$
and denote the induced add one maps for
\(\mathcal{W}\), \(\mathcal{R}\), and \(\mathcal{K}\) by
\(w^+\), \(c^+\), and \(e^+\), respectively.

We have a natural injective map $
\Ker(w^+)  \hookrightarrow  \Ker(v^+)
$ and a surjective map
$
\Coker(u^+)
\twoheadrightarrow
\Coker(w^+).
$
Consequently,
\[
\dim_k\Coker(w^+)
+
\dim_k\Ker(w^+)
\leq
\dim_k\Coker(u^+)
+
\dim_k \Ker(v^+)
.
\]
Since the functors $
\cM_{t_i,-}$, $i=1,2$, satisfy add one kernel-cokernel boundedness, 
this implies that \(\mathcal{W}\) satisfies add one kernel-cokernel boundedness.

We also have a commutative diagram with exact rows
\[
\begin{tikzcd}
0 \arrow[r]
&
\mathcal{W}(F)(r)_q
    \arrow[r]
    \arrow[d,"w^+"']
&
\cM(F)(t_2,r)_q
    \arrow[r]
    \arrow[d,"v^+"]
&
\mathcal{R}(F)(r)_q
    \arrow[r]
    \arrow[d,"c^+"]
&
0
\\
0 \arrow[r]
&
\mathcal{W}(F^+)(r)_q
    \arrow[r]
&
\cM(F^+)(t_2,r)_q
    \arrow[r]
&
\mathcal{R}(F^+)(r)_q
    \arrow[r]
&
0.
\end{tikzcd}
\]
The snake lemma yields an exact sequence
\[
0
\to
\Ker(w^+)
\to
\Ker(v^+)
\to
\Ker(c^+)
\to
\Coker(w^+)
\to
\Coker(v^+)
\to
\Coker(c^+)
\to
0.
\]
It follows that
\[
\begin{aligned}
\dim_k\Coker(c^+)
+
\dim_k\Ker(c^+)
&\leq
\dim_k\Coker(v^+)
+
\dim_k\Ker(v^+)
+
\dim_k\Coker(w^+)
\\
&\leq
\dim_k\Ker(v^+)
+
\dim_k\Coker(v^+)
+
\dim_k\Coker(u^+).
\end{aligned}
\]
Hence, \(\mathcal{R}\) satisfies add one kernel-cokernel boundedness.
An argument dual to that for \(\mathcal{R}\) shows that \(\mathcal{K}\) also satisfies add one kernel-cokernel boundedness.

(3) is follows from (2) and (3).
\end{proof}

\section{Extended persistent Betti numbers}
\label{sec:extended-persistent-betti-numbers}

Extended persistent Betti numbers, introduced in~\cite[Definition~4.6]{JKP26}, are invariants that generalize ordinary persistent Betti numbers. 

In this appendix, we formulate a transformation of $\RR$-indexed chain complexes that realizes an extended persistent Betti number as an ordinary persistent Betti number following \cite{JKP26}. We then show that this transformation preserves translation invariance, add one kernel-cokernel stabilization, and the add one kernel-cokernel boundedness condition. Consequently, provided that the measurability of the resulting persistent Betti number functional is verified separately, Theorem~\ref{thm:clt_main} can also be applied to extended persistent Betti numbers. 

Fix \(j\in\mathbb{Z}\) and \(t\geq 0\). Define a functor
$
S_{j,t}\colon
\Fun
(
\mathbb{R},
\Ch(\Vect_k)
)
\rightarrow
\Fun
(
\mathbb{R},
\Ch(\Vect_k)
)
$
as follows. For
$
C\in
\Fun
(
\mathbb{R},
\Ch(\Vect_k)
), 
$
define
\[
(S_{j,t}C)_q(a)
:=
\begin{cases}
C_q(a), & q\leq j,\\
C_q(a-t), & q\geq j+1
\end{cases}
\]
for each $a \in \RR$. The boundary map in degree \(j+1\) is defined as the composite
\[
C_{j+1}(a-t)
\xrightarrow{\partial}
C_j(a-t)
\xrightarrow{C_j(a-t\leq a)}
C_j(a).
\]
All the other boundary maps are inherited from \(C(a)\) or \(C(a-t)\), as appropriate.
For \(a\leq b\), the structure map of \(S_{j,t}C\) is given in degree \(q\) by
\[
(S_{j,t}C)_q(a\leq b)
:=
\begin{cases}
C_q(a\leq b), & q\leq j,\\
C_q(a-t\leq b-t), & q\geq j+1.
\end{cases}
\]
For a morphism $f \colon C \to D $ in $\Fun
(\mathbb{R},\Ch(\Vect_k))$, we define 
\[
S_{j,t} (f)_q (a) := \begin{cases}
    f_{q,a} \colon C_q (a) \to D_q (a)  &\text{ if }  q \leq j, \\
    f_{q,a-t} \colon C_q(a -t) \to D_q (a-t)  &\text{ if } q\geq j +1
\end{cases} 
\]
for each $a \in \RR$.

For \(C\in
\Fun
(\mathbb{R},\Ch(\Vect_k))\), \(q\in\ZZ\), 
and \(r,s\in\mathbb{R}\), the $q$-th \((r,s)\)-extended persistent Betti number is given by
\[
\beta_{q,\mathrm{ext}}^{r,s}(C)
:=
\begin{cases}
\beta_q^{r,s}(C), & r\leq s,\\
\beta_q^{r,r}(S_{q,r-s}C), & r>s.
\end{cases}
\]
Suppose that \(r>s\). By the definition of \(S_{q,r-s}\), we have
\[
H_q((S_{q,r-s}C)(r))
=
\frac{
Z_q(C(r))
}{
\Image 
(
C_{q+1}(s)
\xrightarrow{\partial}
C_q(s)
\xrightarrow{C_q(s\leq r)}
C_q(r)
)
}.
\]
If the structure map \(C_q(s)\to C_q(r)\) is injective, then
$
\beta_q^{r,r}(S_{q,r-s}C)
=
\dim_k Z_q(C(r))
-
\dim_k B_q(C(s)).
$
Thus, in the setting of filtered chain complexes, this quantity agrees with the \((r,s)\)-extended persistent Betti number introduced in \cite{JKP26}. 
Note that, unlike the usual persistent Betti numbers, we do not impose
the condition \(r\leq s\). See \cite{JKP26} for an interpretation of
this invariant in the context of persistent homology.

Now let
$
\mathcal{A}\colon
\cU
\rightarrow
\Fun
(
\mathbb{R},
\Chnn(\Vect_k)
)
$
be a functor and define 
\begin{equation}\label{eq:trans-functor}
\mathcal{A}^{(q,t)}
:=
S_{q,t}\circ\mathcal{A}.     
\end{equation}



\begin{proposition}
\label{prop:extended-preservation}
Let \(q\in\mathbb{Z}\) and \(t\geq 0\). Then the following statements hold.
\begin{enumerate}
    \item If \(\mathcal{A}\) is translation invariant, then so is 
    \(\mathcal{A}^{(q,t)}\).

    \item If \(\mathcal{A}\) satisfies pointwise  add one kernel-cokernel stabilization, then so does \(\mathcal{A}^{(q,t)}\). 

    \item If \(\mathcal{A}\) satisfies the add one kernel-cokernel  boundedness condition, then so does  \(\mathcal{A}^{(q,t)}\).

\end{enumerate}    
In particular, if $\cA$ is pointwise admissible, then so is  \(\mathcal{A}^{(q,t)}\).
\end{proposition}

\begin{proof}
(1) follows immediately from the definition.
We next prove (2). Let \(P\subseteq\mathbb{R}^d\) be locally finite and let \(r\in\mathbb{R}\). By assumption, both
$
u_{P,r}^{+}\colon
\mathcal{A}_{P,r}
\rightarrow
\mathcal{A}_{P,r}^{+}
$
and
$
u_{P,r-t}^{+}\colon
\mathcal{A}_{P,r-t}
\rightarrow
\mathcal{A}_{P,r-t}^{+}
$
 satisfy kernel-cokernel stabilization. By the definition
of \(\mathcal{A}^{(q,t)}\), each degree of its add one map at \(r\)
is given by the corresponding degree of either \(u_{P,r}^{+}\) or
\(u_{P,r-t}^{+}\). Taking a common upper bound of their stabilization
indices, we conclude that the add one map of \(\mathcal A^{(q,t)}\) at \(r\)
  satisfies kernel-cokernel stabilization.
Since \(r\) was
  arbitrary, \(\mathcal A^{(q,t)}\) satisfies pointwise add one
  kernel-cokernel stabilization.
Finally, suppose that \(\mathcal{A}\) satisfies the add one kernel-cokernel boundedness condition. For every \(F\in\cU\), \(r\in\mathbb{R}\), and \(j\in\mathbb{Z}\), the kernel and cokernel of the degree-\(j\) component of the add one map $ \mathcal{A}^{(q,t)}(F)(r)
\rightarrow
(\mathcal{A}^{(q,t)})^{+}(F)(r)$ are those of either
$(u_{F,r}^{+})_j$ or $(u_{F,r-t}^{+})_j$, according to the degree. The required boundedness estimate therefore follows by taking the pointwise maxima of the bounds at \(r\) and \(r-t\).
\end{proof}